\documentclass{amsart}

\usepackage{amsmath,amssymb,amsthm,units, stmaryrd,stackrel,relsize,bm}

\usepackage{yfonts,upgreek,units,stmaryrd,color,graphicx}
\usepackage[all,cmtip]{xy}
\usepackage{accents}
\usepackage{enumerate}

\usepackage[textsize=footnotesize,color=yellow, bordercolor=white]{todonotes}

\newtheorem{theorem}{Theorem}[section]

\newtheorem{lemma}[theorem]{Lemma}

\newtheorem{claim}[theorem]{Claim}
\newtheorem{subclaim}[theorem]{Subclaim}
\newtheorem*{subclaim*}{Subclaim}
\newtheorem{corollary}[theorem]{Corollary}

\theoremstyle{definition}
\newtheorem{definition}[theorem]{Definition}

\theoremstyle{remark}
\newtheorem{remark}[theorem]{Remark}

\DeclareSymbolFont{AMSb}{U}{msb}{m}{n}
\DeclareMathSymbol{\N}{\mathbin}{AMSb}{"4E}
\DeclareMathSymbol{\Z}{\mathbin}{AMSb}{"5A}
\DeclareMathSymbol{\R}{\mathbin}{AMSb}{"52}
\DeclareMathSymbol{\Q}{\mathbin}{AMSb}{"51}
\DeclareMathSymbol{\I}{\mathbin}{AMSb}{"49}
\DeclareMathSymbol{\C}{\mathbin}{AMSb}{"43}

\usepackage[
    colorlinks=true, 
    citecolor=red,
    linkcolor=blue,
    urlcolor=blue]{hyperref}
\newcommand\ad{{\mathsf{AD}}}

\newcommand\zfc{{\mathsf{ZFC}}}
\newcommand\zf{{\mathsf{ZF}}}
\newcommand\dc{{\mathsf{DC}}}

\newcommand\Ord{{\mathsf{Ord}}}
\newcommand\mc{{\mathsf{MC}}}
\newcommand\od{{\mathsf{OD}}}

\newcommand\crit{\textsc{crit}}
\newcommand\ult[2]{\text{Ult}(#1, #2)}

\newcommand\col{\text{Col}}

\newcommand{\BS}{{}^\omega\omega}

\newcommand{\PI}{\boldsymbol\Pi}
\newcommand{\SIGMA}{\boldsymbol\Sigma}
\newcommand{\DELTA}{\boldsymbol\Delta}
\newcommand{\cQ}{\mathcal{Q}}
\newcommand{\cP}{\mathcal{P}}
\newcommand{\cM}{\mathcal{M}}
\newcommand{\cN}{\mathcal{N}}
\newcommand{\cT}{\mathcal{T}}
\newcommand{\cU}{\mathcal{U}}
\newcommand{\cD}{\mathcal{D}}

\newcommand{\cR}{\mathcal{R}}
\newcommand{\cS}{\mathcal{S}}
\newcommand{\cG}{\mathcal{G}}

\newcommand{\bR}{\mathbb{R}}
\newcommand{\bP}{\mathbb{P}}

\newcommand{\bS}{\mathbb{S}}

\newcommand{\langpm}{\mathcal{L}_{\mathrm{pm}}(\{\dot x_i : i \in \omega\})} 
\newcommand{\lpm}{\mathcal{L}_{\mathrm{pm}}} 
\newcommand{\ctbleset}{\mathcal{P}_{\omega_1}(\mathbb{R})}

\newcommand{\lh}{\ensuremath{\operatorname{lh}}}
\newcommand{\Hom}{\ensuremath{\operatorname{Hom}}}
\newcommand{\Pot}{\mathcal{P}}

\newcommand{\comm}[1]{}

\begin{document}
\title{The consistency strength of long projective determinacy}



\author{Juan P. Aguilera}
\address{Juan P. Aguilera, Institut f\"ur diskrete Mathematik und
  Geometrie, Technische Universit\"at Wien. Wiedner Hauptstrasse 8-10,
  1040 Wien, Austria.} 
\email{aguilera@logic.at}
\author{Sandra M\"uller} 
\address{Sandra M\"uller, Kurt G\"odel
  Research Center, Institut f\"ur Mathematik, UZA 1, Universit\"at
  Wien. Augasse 2-6, 1090 Wien, Austria.}
\email{mueller.sandra@univie.ac.at} 
\thanks{The second-listed author, formerly known as Sandra Uhlenbrock, was
  partially supported by FWF grant number P 28157.}

\date{\today}

\begin{abstract}
  We determine the consistency strength of determinacy for projective
  games of length $\omega^2$. Our main theorem is that
  $\PI^1_{n+1}$-determinacy for games of length $\omega^2$ implies the
  existence of a model of set theory with $\omega + n$ Woodin
  cardinals. In a first step, we show that this hypothesis
  implies that there is a countable set of reals $A$ such that
  $M_n(A)$, the canonical inner model for $n$ Woodin cardinals
  constructed over $A$, satisfies $A = \bR$ and the Axiom of
  Determinacy. Then we argue how to obtain a model with $\omega + n$
  Woodin cardinal from this.
  
  We also show how the proof can be adapted to investigate the
  consistency strength of determinacy for games of length $\omega^2$
  with payoff in $\Game^\bR \PI^1_1$ or with $\sigma$-projective
  payoff.
\end{abstract}

\clearpage
\maketitle
\setcounter{tocdepth}{1}

\section{Introduction}\label{SectIntro}
We study the consistency strength of determinacy for games of length
$\omega^2$ with payoff in various pointclasses $\Gamma$.
Specifically, given a set $A\subset \omega^{\omega^2}$, i.e., a set of
sequences of natural numbers of length $\omega^2$, with
$A \in \Gamma$, consider the following game:
\[ \begin{array}{c|cccccccc} \mathrm{I} & n_0 & & n_2 & &\hdots &
    n_\omega & & \hdots \\ \hline
    \mathrm{II} & & n_1 & & n_3 & \hdots & & n_{\omega+1} & \hdots 
   \end{array} \]
 Players I and II alternate turns playing natural numbers to produce
 some
 $x = (n_0, n_1, \dots, n_\omega, n_{\omega+1}, \dots) \in
 \omega^{\omega^2}$.  Player I wins such a run $x$ of
 the game if, and only if, $x \in A$; otherwise, Player II wins. We
 study the strength of the statement that games of this form are
 determined, i.e., that one of 
 the players has a winning strategy. For all nontrivial classes
 $\Gamma$, this question is independent of Zermelo-Fraenkel set theory
 with the Axiom of Choice ($\zfc$); for some of them, however, it
 is known to follow from natural strengthenings of $\zfc$, namely, from assumptions
 on the existence of \emph{large cardinals}. 

 Recall that the projective subsets of a Polish space are those
 obtainable from Borel sets in finitely many stages by applying
 complements and projections from a finite power of the space.  We are
 mainly interested in the case where $\Gamma$ is a projective
 pointclass, i.e., $\PI^1_n$ for some natural number $n$, but we will
 also consider the cases in which $\Gamma$ is equal to
 $\Game^\bR \PI^1_1$ or a $\sigma$-projective pointclass.\footnote{The
   pointclass of all $\sigma$-projective sets is the smallest
   pointclass closed under complements, countable unions, and
   projections, where countable unions refer to sets which are subsets
   of the same product space. Moreover, we as usual identify $\bR$ and
   $\BS$.}
 
 The study of games of length $\omega^2$ is motivated by the folklore
 result that projective determinacy for games of length $\omega$
 implies projective determinacy for games of length $\alpha$, for any
 $\alpha<\omega^2$. It is also not difficult to see that
 $\PI^1_{n+1}$-determinacy of length $\omega^2$ follows from analytic
 determinacy for games of length $\omega\cdot(\omega+n)$, for 
 natural numbers $n$.\footnote{In fact, an argument as in
   \cite[Proposition 2.7]{AMS} with a more careful analysis of the
   complexity of the payoff sets (using projective determinacy) shows
   that these two determinacy hypotheses are equivalent.}

 Analytic determinacy for games of length $\omega \cdot (\omega + n)$
 was proved by Neeman \cite[Theorem 2A.3]{Ne04} from a large cardinal
 hypothesis. Specifically, he assumed the existence of a weakly
 iterable model of set theory with $\omega + n$ Woodin cardinals,
 along with a sharp for the model. In fact, he proved this result for
 analytic games of any fixed countable length
 $\omega \cdot (\theta + 1)$, from corresponding assumptions. This
 naturally yields the question whether his results are optimal, i.e.,
 whether the determinacy of these long games implies the existence of
 models with certain numbers of Woodin cardinals. In light of this, we
 prove the following theorem:
  
  \begin{theorem}\label{TheoremLongIntro}
    Suppose that $\PI^1_{n+1}$-determinacy for games of length
    $\omega^2$ holds and let $x \in\BS$ be arbitrary. Then, there is a
    proper class model $M$ of $\zfc$ with $\omega+n$ Woodin cardinals
    such that $x \in M$.
  \end{theorem}

  The Woodin cardinals of the model we construct in the proof of
  Theorem \ref{TheoremLongIntro} are in reality countable. Moreover,
  the model is a premouse, i.e., fine structural, and we can in fact
  construct it such that it is active, i.e., it has a sharp on top. As
  a corollary of the theorem and the results from \cite{Ne04}, we
  obtain the following equiconsistencies, which also connect
  projective determinacy for games of length $\omega^2$ to projective
  determinacy for games on reals.\footnote{We would like to thank the
    referee for asking whether \eqref{eq:3} and \eqref{eq:4} could be
    added to Corollary \ref{TheoremLongPD}; see also \cite{AgMu}.}

  \begin{corollary}\label{TheoremLongPD}
    The following schemata are equiconsistent:
    \begin{enumerate}
    \item $\zfc$ + $\{`` \PI^1_n$-determinacy for games of length
      $\omega^2$ on $\omega$'' $: n\in\omega\}$. \label{eq:2}
    \item $\zf$ + $\dc$ + $\ad$ + $\{``$there are $n$ Woodin
      cardinals''$: n\in\omega\}$. \label{eq:1.5}
    \item $\zfc$ + $\{``$there are $\omega+n$ Woodin
      cardinals''$: n\in\omega\}$. \label{eq:1}
    \item $\zf$ + $\dc$ + $\ad$ + $\{`` \PI^1_n$-determinacy for games of length
      $\omega$ on $\bR$'' $: n\in\omega\}$.\label{eq:3}
    \item $\zf$ + $\dc$ + $\ad$ +
      $\{`` V^{\col(\omega,\bR)} \vDash \PI^1_n$-determinacy for games
      of length $\omega$ on $\omega\text{''} : n\in\omega\}$.\label{eq:4}
    \end{enumerate}
  \end{corollary}

  The proof of Theorem \ref{TheoremLongIntro} has two main parts: in
  the first one, the hypothesis is shown to imply that for a closed
  and unbounded set of countable sets of reals $A$, there is a fine
  structural model of the Axiom of Determinacy with $n$ Woodin
  cardinals whose set of real numbers is precisely $A$. In the second
  part we use a Prikry-like partial order to force over these models
  and obtain via a translation procedure an infinite sequence of
  Woodin cardinals below the already existing ones. The proof of
  Corollary \ref{TheoremLongPD} is sketched in the final section of
  this paper.\\

  \paragraph{\textbf{Background}} The situation for games of length
  $\omega$ is well understood: There is a tight connection beween
  determinacy for these games and the existence of inner models with
  large cardinals. Martin \cite{Ma75} showed that Borel games are
  determined in $\zfc$. Contrary to that, determinacy for $\SIGMA^1_1$
  (i.e., analytic) games cannot be proved in $\zfc$ alone. By theorems
  of Martin \cite{Ma70} and Harrington \cite{Ha78} $\SIGMA^1_1$ games
  are determined if, and only if, $x^\sharp$ exists for every
  $x\in\BS$. Martin and Steel \cite{MaSt89} proved, for each
  $n\in\omega$, that the existence of $n$ Woodin cardinals with a
  measurable cardinal above implies the determinacy of all
  $\SIGMA^1_{n+1}$ sets. Woodin (unpublished) improved this for odd
  $n$ by showing that the existence of $M_n^\sharp(x)$, the canonical
  active $\omega_1$-iterable inner model with $n$ Woodin cardinals
  constructed over $x$, for all reals $x$ suffices to show
  $\SIGMA^1_{n+1}$ determinacy. Afterwards, Neeman improved this in
  \cite{Ne95} even further and showed that for all $n$, the existence
  of $M_n^\sharp(x)$ for all reals $x$ implies determinacy of all
  $\Game^n({<}\omega^2-\PI^1_1)$ sets. Concerning the other direction,
  Woodin (see \cite{MSW}) showed that if $\SIGMA^1_{n+1}$ games are
  determined, then $M_n^\sharp(x)$ exists for all reals $x$, thus
  establishing a level-by-level characterization of projective
  determinacy in terms of the existence of inner models with large
  cardinals. Similar characterizations are known for
  $\sigma$-projective games of length $\omega$ (see \cite{Ag} and
  \cite{AMS}). Determinacy for games of length $\omega$ with payoff in
  $\Game^\bR \PI^1_1$ is equivalent to $\ad^{L(\bR)}$ (see
  \cite{MaSt08}).

  Woodin showed that the existence of $\omega^{2}$ Woodin cardinals
  under choice is equiconsistent with $\ad^+ + \dc$ and the existence
  of a normal fine measure on $\ctbleset$ (see Remark 9.98 in
  \cite{Wo10}). By this and a result of Trang (see \cite[Theorem
  2.3.11]{Tr13}, case $\alpha = 2$), determinacy of analytic games of
  length $\omega^{3}$ implies the consistency of $\omega^{2}$ Woodin
  cardinals under choice. In fact, a similar result holds for
  $\omega^\alpha$ Woodin cardinals and determinacy of analytic games
  of length $\omega^{1+\alpha}$ for all $1 < \alpha < \omega$ and all
  limit ordinals $\omega \leq \alpha < \omega_1$. For details see
  Theorems 2.2.1 and 2.3.11 in \cite{Tr13}, as well as \cite{Tr14} and
  \cite{Tr15}.\\

  \paragraph{\textbf{Further results}} The proof of Theorem
  \ref{TheoremLongIntro} uses game arguments based on techniques from
  Martin and Steel \cite{MaSt08} and \cite{MSW}, tracing back to
  arguments of H. Friedman \cite{Fr71} and Kechris and Solovay
  \cite{KS85} as well as inner model theoretic methods based on the
  unpublished notes \cite{St}. The method of the proof of Theorem
  \ref{TheoremLongIntro} can also be used to show the following two
  generalizations:

 \begin{theorem}\label{TheoremLongomega+omegaIntro}
   Suppose that $\Game^\bR\PI^1_1$-determinacy for games of length
   $\omega^2$ holds and let $x \in\BS$ be arbitrary. Then, there is a
   proper class model $M$ of $\zfc$ with $\omega+\omega$ Woodin
   cardinals such that $x \in M$.
 \end{theorem}

 As above, it is not hard to show that determinacy of analytic games
 of length $\omega\cdot(\omega+\omega)$ implies determinacy of
 $\Game^\bR \PI^1_1$ games of length $\omega^2$. Our methods also
 generalize to games with $\sigma$-projective payoff and premice of
 class $S_\alpha$ (see \cite{Ag} and \cite[Definition 4.1]{AMS}).

\begin{theorem}\label{TheoremLongSalphaIntro}
  Suppose that $\sigma$-projective determinacy for games of length
  $\omega^2$ holds and let $x \in\BS$ be arbitrary. Then, for every
  $\alpha < \omega_1$, there is a proper class model $M$ of $\zfc$
  with $\omega$ Woodin cardinals with supremum $\lambda$ which is of
  class $S_\alpha$ above $\lambda$ and such that $x \in M$.
\end{theorem}

\paragraph{\textbf{Outline}} In Section \ref{SectionPreliminaries}, we
establish conventions and recall some known facts about extender
models which will be used later on. Focusing first on the case
$n = 1$, Section \ref{SectADinM1} contains the main argument and shows
that determinacy for games of length $\omega^2$ with $\PI^1_2$ payoff
implies the existence of fine structural models of the Axiom of
Determinacy with one Woodin cardinal. In Section \ref{SectionDC}, we
argue that this fine structural model satisfies $\dc$ and $\ad^+$. In
Section \ref{SectOmega+1WdnsFromAD}, we show how to obtain a model of
$\zfc$ with $\omega+1$ Woodin cardinals from this model. Finally, in
Section \ref{SectionRemaining} we sketch the proof of Corollary
\ref{TheoremLongPD} and explain how to carry out the modifications
needed to prove Theorems \ref{TheoremLongomega+omegaIntro} and
\ref{TheoremLongSalphaIntro}. \\

We would like to thank John Steel for valuable discussions related to
this paper, in particular to Lemma \ref{LemmaRealsStationary}, and for
making the unpublished notes \cite{St} available to us. Moreover, we
would like to thank Grigor Sargsyan for his helpful comments on an
earlier version of this manuscript. Finally, we would like to thank
the referee for carefully reading our paper and making several
valuable suggestions.

\section{Preliminaries}\label{SectionPreliminaries}

For basic set theoretic definitions and results we refer to
\cite{Ka08} and \cite{Mo09}. Moreover, we work with canonical, fine
structural models with large cardinals, called \emph{premice}. We
refer the reader to e.g., \cite{St10} for an introduction, and to
\cite{MS94}, \cite{SchStZe02}, and \cite{St08b} for additional
background. In particular, we will use Mitchell-Steel indexing for
extender sequences and the notation from \cite{St10}. As in
\cite{St08b} we will consider relativized premice constructed over
arbitrary transitive sets $X$. Let
$\lpm = \{\dot \in, \dot E, \dot F, \dot X\}$ denote the language of
relativized premice, where $\dot E$ is the predicate for the extender
sequence, $\dot F$ is the predicate for the top extender, and $\dot X$
is the predicate for the set over which we construct the premouse.

For a transitive set $X$, we say an \emph{$X$-premouse}
$M = (J_\alpha^{\vec{E}}, \in, \vec{E}, E_\alpha, X)$ for
$\vec{E} = (\dot E)^M$, $E_\alpha = (\dot F)^M$, and $X = (\dot X)^M$
is \emph{active} if $E_\alpha \neq \emptyset$. Otherwise, we say $M$
is \emph{passive}. We let
$M | \gamma = (J_\gamma^{\vec{E}}, \in, \vec{E} \upharpoonright
\gamma, E_\gamma, X)$ for $\gamma \leq M \cap \Ord$. Moreover, we
write $M || \gamma$ for the passive initial segment of $M$ of height
$\gamma$, i.e.
$M || \gamma = (J_\gamma^{\vec{E}}, \in, \vec{E} \upharpoonright
\gamma, \emptyset, X)$, for $\gamma \leq M \cap \Ord$. In particular,
$M||\Ord^M$ denotes the premouse $N$ which agrees with $M$ except that
we let $(\dot F)^N = \emptyset$. We say an ordinal $\eta$ is a
\emph{(strong) cutpoint} of $M$ if there is no extender $E$ on the
$M$-sequence with $\crit(E) \leq \eta \leq \lh(E)$.

An arbitrary $X$-premouse might not satisfy the Axiom of Choice, but
it can be construed as an ordinary premouse which satisfies the Axiom
of Choice if a well-order of $X$ is added generically. In particular,
every $X$-premouse $M$ is \emph{well-ordered mod $X$}, i.e., for every
set $Y \in M$ there is an ordinal $\alpha$ and a surjection $h$ in $M$
such that $h \colon X \times \alpha \twoheadrightarrow Y$. In this
article we shall mainly be interested in $X$-premice which satisfy the
Axiom of Determinacy (and hence not the full Axiom of Choice), where
$X \in \ctbleset$, i.e., $X$ is a countable set of reals.

Throughout this article, we work under the assumption that all premice
are tame, i.e., that there is no extender on the sequence of a
premouse overlapping a Woodin cardinal. This results in no loss of
generality, as otherwise the conclusions of the theorems in Section
\ref{SectIntro} hold.

In the rest of this section we summarize some facts about premice
which we are going to need later. Most of these can be found in
\cite{MSW} and \cite{Uh16} for $x$-premice for a real $x$ and they
straightforwardly generalize to $X$-premice for arbitrary countable
sets of reals $X$. Fix such a set $X \in \ctbleset$ for the rest of this section.

\begin{definition}
  Let $M$ be an $X$-premouse and let $\alpha<\omega_1^L$. Then we say
  that $M$ is \emph{$\alpha$-small} if, and only if, for every ordinal
  $\kappa \leq M \cap \Ord$ such that there is an extender with
  critical point $\kappa$ on the $M$-sequence, in $M | \kappa$ there
  is no set of ordinals $W$ of order-type $\alpha$ such that every
  ordinal in $W$ is a Woodin cardinal in $M | \kappa$.
\end{definition}

If it exists and is unique, we denote the $\omega_1$-iterable,
countable, sound $X$-premouse which is not $\alpha$-small, but all of
whose proper initial segments are $\alpha$-small, by
$M_\alpha^\sharp(X)$. In this case $M_\alpha(X)$ denotes the result of
iterating the top most measure of $M_\alpha^\sharp(X)$ and its images
out of the universe.

We will now argue that under certain conditions there is a comparison
lemma for $n$-small $X$-premice. This will be used in the proof of
Theorem \ref{TheoremADM1}. For premice over a real $x$ this can be
found for example in \cite[Lemma 2.11]{MSW}; the argument there
generalizes to premice over countable sets of reals. For the reader's
convenience we will briefly sketch the main ideas and recall the
statements. The following notion will be important in what follows to
ensure that $\cQ$-structures exist.

\begin{definition}\label{def:notdefWdn}
  Let $M$ be a sound $X$-premouse and let $\delta$ be a cardinal in $M$ or
  $\delta = M \cap \Ord$. We say that \emph{$\delta$ is not definably
    Woodin over $M$} if, and only if, there exists an ordinal
  $\gamma \leq M \cap \Ord$ such that $\gamma \geq \delta$ and either
  \begin{enumerate}[$(i)$]
  \item over $M|\gamma$ there exists an $r\Sigma_n$-definable set
    $A \subset \delta$ for some $n<\omega$ such that for no
    $\kappa < \delta$ do the extenders on the $M$-sequence witness
    that $\kappa$ is strong up to $\delta$ with respect to $A$, or
\item $\rho_n(M|\gamma) < \delta \text{ for some } n < \omega.$
  \end{enumerate}
\end{definition} 

For several iterability arguments to follow we need our premice to
satisfy the following property, which, as a fine structural argument shows,
is preserved during iterations. This in turn ensures that
$\cQ$-structures exist in iterations of a premouse $M$ satisfying it. 

\begin{definition}\label{def:nodefWdns}
  Let $M$ be a sound $X$-premouse. We say $M$ has \emph{no definable Woodin
    cardinals} if, and only if, for all $\delta \leq M \cap \Ord$ we
  have that $\delta$ is not definably Woodin over $M$.
\end{definition}

The proof of \cite[Lemma 2.11]{MSW} generalizes to $X$-premice and
shows the following lemma. Here the hypothesis that $M_n^\sharp(x)$
exists for all $x \in \bR$ is used to compare the countable premice
$M$ and $N$ inside the model $M_n^\sharp(x)$ for a real $x$ coding $M$
and $N$. The fact that the iteration strategies for $M$ and $N$ are
guided by $\cQ$-structures ensures that their restriction to trees in
$H^{M_n^\sharp(x)}_{\delta_x}$, where $\delta_x$ is the least Woodin
cardinal in $M_n^\sharp(x)$, is in $M_n^\sharp(x)$ and hence
$\omega_1^V$-iterability suffices for comparison in this situation.

\begin{lemma}\label{lem:comomega1it}
  Suppose that $M_n^\sharp(x)$ exists for all $x \in \bR$. Let $M$ and
  $N$ be countable $\omega_1$-iterable sound $X$-premice such that
  every proper initial segment of $M$ and $N$ is $n$-small and they
  both do not have definable Woodin cardinals. Then there are iterates
  $M^*$ and $N^*$ of $M$ and $N$ respectively such that one of the
  following holds:
\begin{enumerate}
\item $M^*$ is an initial segment of $N^*$ and there is no drop on the
  main branch in the iteration from $M$ to $M^*$,
\item $N^*$ is an initial segment of $M^*$ and there is no drop on the
  main branch in the iteration from $N$ to $N^*$.
\end{enumerate}
\end{lemma}

      In the statement of Lemma \ref{lem:comomega1it} the assumption
      that $M$ and $N$ do not have definable Woodin cardinals can be
      replaced by the assumption that $M$ and $N$ do not have Woodin
      cardinals; see the remark after the proof of Lemma 2.11 in
      \cite{MSW}.

      There is a variant of this lemma for the case that one of the
      premice is only $\Pi^1_n$-iterable as introduced in Definitions
      1.4 and 1.6 in \cite{St95}. In this case, the last model of the
      iteration tree on that premouse need not be fully well-founded,
      but the argument from \cite[Lemma 2.2]{St95} (see also Corollary
      2.15 in \cite{MSW}) yields that we still have a comparison
      lemma.

\begin{lemma}\label{lem:coitPi12omega1}
  Suppose that $M_n^\sharp(x)$ exists for all $x \in \bR$. Let $M$ and
  $N$ be countable $n$-small sound solid $X$-premice which both do
  not have Woodin cardinals. Moreover, assume that $M$ is
  $\omega_1$-iterable and $N$ is $\Pi^1_{n+1}$-iterable. Then there is
  an iteration tree $\cT$ on $M$ and a putative \footnote{We say that
    $\cU$ is a \emph{putative iteration tree} if it satisfies all
    properties of an iteration tree, but we allow the last model, if
    it exists, to be ill-founded.} iteration tree $\cU$ on $N$ of
  length $\lambda+1$ for some limit ordinal $\lambda$ such that one of
  the following holds:
\begin{enumerate}
\item $\cM_\lambda^\cT$ is an initial segment of $\cM_\lambda^\cU$ and
  there is no drop on the main branch through $\cT$. In this case
  $\cM_\lambda^\cU$ need not be fully well-founded, but it is
  well-founded up to $\cM_\lambda^\cT \cap \Ord$.
\item $\cM_\lambda^\cU$ is an initial segment of $\cM_\lambda^\cT$ and
  there is no drop on the main branch through $\cU$. In this case
  $\cM_\lambda^\cU$ is fully well-founded and $\cU$ is an iteration
  tree.
\end{enumerate}
\end{lemma}

Finally, note that by a standard argument $\PI^1_2$-determinacy, or
equivalently that $M^\sharp_1(x)$ exists and is
$\omega_1$-iterable for all $x \in \mathbb{R}$, implies that
$M_1^\sharp(X)$ exists and is $\omega_1$-iterable for any countable
set of reals $X$.

\section{Models of the Axiom of Determinacy with a Woodin
  Cardinal} \label{SectADinM1} Recall that
$\mathcal{P}_{\omega_1}(\mathbb{R})$ denotes the set of all countable
sets of reals. We consider models of the form $M_n(A)$, where
$n \in \omega$ and $A \in \mathcal{P}_{\omega_1}(\mathbb{R})$. It is
not hard to see that---provided they exist---many of these structures
are models of the Axiom of Choice. Our first theorem shows that, if
games of length $\omega^2$ with $\PI^1_{n+1}$ payoff are determined,
then many of these structures are models of the Axiom of Determinacy
and we can in addition have that $M_n(A) \cap \bR = A$.
 
\begin{theorem}\label{TheoremADM1}
  Let $n \in \omega$ and suppose that determinacy for $\PI^1_{n+1}$ games
  of length $\omega^2$ holds. Then, there is a club
  $\mathcal{C} \subset \mathcal{P}_{\omega_1}(\mathbb{R})$ such that
  for all $A \in \mathcal{C}$, $M^\sharp_n(A)$ exists, is
  $\omega_1$-iterable, $M^\sharp_n(A) \cap \bR = A$, and
  \[ M_n(A) \models\zf + \ad.\] 
\end{theorem}

To simplify the notation we will from this point on only consider the
case $n=1$. The general case $n \in \omega$ can be shown by
straightforward modifications of the proof we give for $n=1$
below. For the rest of this section assume that $V$ is a model of
$\zfc$ and Projective Determinacy, i.e., that $M_n^\sharp(x)$ exists
for all reals $x$. Note that the hypothesis of Theorem
\ref{TheoremLongIntro}, determinacy for $\PI^1_{n+1}$ games of length
$\omega^2$, trivially implies Projective Determinacy since after the
first $\omega$ moves the following rounds can be used to ``play
witnesses for projections''. Whenever we are assuming a stronger
determinacy hypothesis, we will point it out explicitly.

Before moving on to the proof of Theorem \ref{TheoremADM1}, we recall
some basic model-theoretic facts we will need later. In addition to
the language of premice,
$\lpm = \{\dot \in, \dot E, \dot F, \dot X\}$, we will now also
consider the language $\langpm$ resulting from enhancing $\lpm$ with
constants $\dot x_i$, for $i \in \omega$. For an $\langpm$-model
$\cM = (M,\in, \dot E^\cM,\dot{F}^\cM,\dot X^\cM,\{\dot x_i^\cM \colon
i < \omega\})$ we write $\cM \upharpoonright \lpm$ for the restriction
$(M,\in,\dot E^\cM, \dot F^\cM, \dot X^\cM)$ of the model $\cM$ to the
smaller language $\lpm$.

Let
$\mathcal{M} = (M,\in, \dot E^\cM, \dot F^\cM, \dot X^\cM, \{x_i
\colon i \in \omega\})$ for $x_i = \dot x_i^\cM$, $i \in \omega$, be a
model in the enhanced language, so in particular
$\{x_i \colon i\in\omega\} \subseteq M$. The \emph{definable closure}
of $\{x_i \colon i\in\omega\}$ in $\mathcal{M} \upharpoonright \lpm$
is defined to be the submodel
\[ (\bar M, \in, \dot E^\cM \cap \bar{M}, \dot F^\cM \cap \bar{M},
  \dot X^\cM \cap \bar{M}) \] of $\mathcal{M} \upharpoonright \lpm$
where $\bar{M}$ consists of all $a\in M$ such that for some $k<\omega$
and some $\lpm$-formula $\phi(v,v_0, \dots, v_k)$,
\begin{align*}
\cM \upharpoonright \lpm \models \text{``$\phi[a, x_0, \dots, x_k]$ and there is a
  unique $x$ such that $\phi[x, x_0, \dots, x_k]$.''}
\end{align*}

For sufficiently nice theories $T$, the definable closure of a model
of $T$ does not depend on the model itself but only on the theory. 

\begin{lemma}\label{LemmaNDoesNotDependOnM}
  Suppose $T\supset \zf$ is a complete, consistent theory in the
  language $\langpm$ with the property that whenever
\[\mathcal{M} = (M,\in, \dot E^\cM, \dot F^\cM, \dot X^\cM, \{\dot
  x_i^\cM \colon i \in \omega\})\] 
is a model of $T$ and $\mathcal{N}^\cM$ is the definable closure of
$\{\dot x_i^\cM \colon i\in\omega\}$ in $\mathcal{M}\upharpoonright\lpm$, then
\begin{enumerate}
\item for each $i\in\omega$, $\cM\models \dot x_i^\cM \in
  \mathbb{R}$, \label{eq:LemmaNDoesNotDependOnM1} 
\item $\mathcal{N}^\cM \prec \mathcal{M} \upharpoonright
  \lpm$. \label{eq:LemmaNDoesNotDependOnM2}
\end{enumerate}
Then $\mathcal{N}^\cM$ does not depend on $\mathcal{M}$, i.e., if
$\cP$ is another model of $T$ with properties
\eqref{eq:LemmaNDoesNotDependOnM1} and
\eqref{eq:LemmaNDoesNotDependOnM2}, then $\cN^\cM$ and $\cN^\cP$ are
isomorphic.
\end{lemma}
\begin{proof}
Let $T$ be as in the statement and let
\[\mathcal{M} = (M,\in, \dot E^\cM, \dot F^\cM, \dot X^\cM,\{\dot
  x^\cM_i \colon i\in\omega\})\]
and
\[\mathcal{P} = (P,\in, \dot E^\cP, \dot F^\cP, \dot X^\cP,\{\dot
  x^\cP_i \colon i\in\omega\})\] be two models of $T$. Since $T$ is
complete, $\mathcal{M}$ and $\mathcal{P}$ are elementarily equivalent
with respect to the language $\langpm$. If $\mathcal{N}^\cM$ and
$\mathcal{N}^\cP$ are the respective definable closures of
$\cM \upharpoonright \lpm$ and $\cP \upharpoonright \lpm$, the natural
function $\rho$ given by
\begin{align*}
  \text{the unique $a \in M$ such that $\cM \upharpoonright \lpm
  \models \phi[a, \dot x_1^\cM, \dots, \dot x_k^\cM]$ } \\ \mapsto
  \text{ the unique $b \in P$ such that $\cP \upharpoonright \lpm
  \models \phi[b, \dot x_1^\cP, \dots, \dot x_k^\cP]$ } 
\end{align*} 
for some $k < \omega$ and some $\lpm$-formula
$\phi(x, v_1, \dots, v_k)$ is an isomorphism from $\mathcal{N}^\cM$ to
$\mathcal{N}^\cP$. This follows from the following observations:
\begin{enumerate}
\item Since $\dot x_i$ are constants interpreted by reals and $T$ is
  complete, they have the same interpretation in $\mathcal{M}$ and
  $\mathcal{P}$, i.e., for all $i \in \omega$, $\dot x_i^\cM = \dot x_i^\cP$.
\item If $x \in \mathcal{N}^\cM$, then there is an $\lpm$-formula
  $\psi(v,v_1,\dots,v_k)$ such that $\cM \upharpoonright \lpm \models
  \psi[x, \dot x_1^\cM, \dots, \dot x_k^\cM]$ and $x$ is unique with
  this property in $\cM \upharpoonright \lpm$. Thus, if $x \in \dot
  E^\cM$, 
\begin{align*}
  \cM \models \text{``$x$ is the unique element satisfying
  $\psi[x,\dot x_1^\cM,\dots, \dot x_k^\cM]\wedge x\in \dot E^\cM$,''}
\end{align*}
and so by considering the $\langpm$-sentence ``there exists a unique
$x$ with $\psi(x, \dot x_1,\dots, \dot x_k) \wedge x \in \dot E$'',
\begin{align*}
  \cP \models \text{``$\rho(x)$ is the unique element satisfying
  $\psi[\rho(x), \dot x_1^\cP, \dots, \dot x_k^\cP] \wedge \rho(x) \in
  \dot E^\cP$'',}
\end{align*}
hence $\rho(x) \in \dot E^\cP$. If $x \not\in \dot E^\cM$, we have
$x \not\in \dot E^\cP$ by the same argument.
\end{enumerate}
The argument for the other predicates is analogous; hence, all
predicates are interpreted the same way. Therefore, $\mathcal{N}^M$
and $\mathcal{N}^P$ are indeed isomorphic.
\end{proof}

We will now set a general context in which we prove the following two
lemmas, as we want to apply them in both the proof of Theorem
\ref{TheoremADM1} and the proof of Lemma \ref{LemmaRealsStationary}
below, for different formulae $\varphi$.

\begin{definition}\label{def:phiwitness}
  Let $X \in \ctbleset$ and let $\cN$ be a countable $X$-premouse. Let
  $\varphi$ be an $\lpm$-formula without free variables.
  \begin{enumerate}
  \item We say $\cN$ is a \emph{$\varphi$-witness} if, and only if,
    $\cN$ is 1-small, sound, and solid, $\cN \vDash \zf$, there are no
    Woodin cardinals in $\cN$, and $\cN \vDash \varphi$.
  \item We say $\cN$ is a \emph{minimal $\varphi$-witness} if, and
    only if, $\cN$ is a $\varphi$-witness and no proper initial
    segment of $\cN$ is a $\varphi$-witness, i.e., whenever $\cP$ is a
    proper initial segment of $\cN$ satisfying $\zf$ + ``there are no
    Woodin cardinals'', then $\cP \nvDash \varphi$.
  \end{enumerate}
\end{definition}

\begin{lemma}\label{lem:Pi12omega1itGeneral} 
  Let $X \in \ctbleset$ and suppose that $\cN$ is a countable
  $\Pi^1_2$-iterable $X$-premouse which is a minimal $\varphi$-witness
  for some $\lpm$-formula $\varphi$. Moreover, assume that there is
  another countable $X$-premouse $\cM$ which is a $\varphi$-witness
  and $\omega_1$-iterable. Then $\cN$ is in fact $\omega_1$-iterable.
\end{lemma}

\begin{proof} Let $x$ be a real coding $\cM$ and
  $\cN$ and consider the coiteration of $\cM$ and $\cN$ inside
  $M_1^\sharp(x)$ in the sense of Lemma \ref{lem:coitPi12omega1} using
  $\omega_1$-iterability for $\cM$ and $\Pi^1_2$-iterability for
  $\cN$. Let $\cT$ and $\cU$ be the resulting iteration trees on $\cM$
  and $\cN$ with final models $\cM^*$ and $\cN^*$ respectively. We
  will show that $\cN$ cannot win this comparison, i.e., 
  $\cN^* \unlhd \cM^*$ and there is no drop on the main branch through
  $\cU$. This implies that $\cN$ is elementarily embeddable into the
  $\omega_1$-iterable premouse $\cM^*$ and thus $\omega_1$-iterable.

  Assume first that $\cM^* = \cN^*$, there is no drop on the main
  branch through $\cT$, and there is at least one drop on the main
  branch through $\cU$. Then $\cM^*$ is (by elementarity) a model of
  $\zf$, contrary to the fact that
  $\rho_\omega(\cN^*) < \cN^* \cap \Ord$.

  Finally assume, again towards a contradiction, that
  $\cM^* \lhd \cN^*$ and there is no drop on the main branch through
  $\cT$. The model $\cN^*$ need not be fully well-founded, but this does
  not affect the rest of the argument as we shall work in the
  well-founded part of $\cN^*$.

  Notice that $\cM^*$ is a proper initial segment of $\cN^*$ which (by
  elementarity) satisfies $\zf$, ``there are no Woodin cardinals,''
  and $\varphi$. Therefore, it cannot be that there is no drop in
  model on the main branch through $\cU$, by elementarity and the
  minimality of $\cN$. Assume for simplicity that there is exactly one
  drop in model on the main branch through $\cU$, say, at level
  $\beta+1 < \lambda$ (the general case is similar: if there is more
  than one drop, we repeat the argument). Using the notation from
  \cite[Section 3.1]{St10}, the fact that there is a drop in model at
  stage $\beta+1$ implies that $\cM_{\beta+1}^*$ is a proper initial
  segment of $\cM_\xi^\cU$, where $\xi$ is the $U$-predecessor of
  $\beta+1$ and $\cM_{\beta+1}^*$ is the model to which the next
  extender on the main branch through $\cU$ is applied. So by
  elementarity between $\cM_{\beta+1}^*$ and $\cN^*$, there is an
  ordinal $\alpha^*$ witnessing the failure of the minimality property
  for $\cM_{\beta+1}^*$, i.e., the following hold:
  \begin{enumerate}
  \item $\alpha^* < \cM_{\beta+1}^* \cap \Ord < \cM_\xi^\cU \cap \Ord$,
  \item $\cM_{\beta+1}^*|\alpha^*$ is a model of $\zf$ with no Woodin
    cardinals, and
  \item $\cM_{\beta+1}^*|\alpha^* \vDash \varphi.$
  \end{enumerate}
  But $\cM_{\beta+1}^*$ is an initial segment of $\cM_\xi^\cU$, so the
  same holds for $\cM_\xi^\cU | \alpha^*$. Now by elementarity
  again---this time between $\cN$ and $\cM_\xi^\cU$---this failure of
  the minimality property also holds for $\cN$, contradicting the fact
  that $\cN$ is a minimal $\varphi$-witness.
\end{proof}

\begin{lemma} \label{lem:CommonIterate} Let $X \in \ctbleset$ and let
  $\cM$ and $\cN$ be $\omega_1$-iterable countable $X$-premice which
  are minimal $\varphi$-witnesses for some $\lpm$-formula
  $\varphi$. Then $\cM$ and $\cN$ have a common iterate and on both
  sides of the iteration there is no drop in model on the main branch
  through the iteration
  tree. 
\end{lemma}

\begin{proof}
  Let $\cT$ and $\cU$ be the iteration trees of length $\lambda +1$
  for some ordinal $\lambda$ on $\cM$ and $\cN$ respectively obtained
  from a successful comparison in the sense of Lemma
  \ref{lem:comomega1it}. Write $\cM^* = \cM_\lambda^\cT$ and
  $\cN^* = \cM_\lambda^\cU$ for the last models of the iteration
  trees. We cannot have $\cM^* \lhd \cN^*$, by the argument of Lemma
  \ref{lem:Pi12omega1itGeneral}. Similarly, the alternative
  $\cN^* \lhd \cM^*$ leads to a contradiction, so we must have
  $\cN^* = \cM^*$.
  
  Only one side of the comparison can drop; assume that there is a
  drop in model on the main branch through $\cU$. The case that the
  main branch through $\cT$ drops is analogous. As in the proof of
  Lemma \ref{lem:Pi12omega1itGeneral}, we assume for simplicity that
  there is exactly one drop in model along the main branch through
  $\cU$, say at stage $\beta+1 < \lambda$; the general case is dealt
  with similarly by repeating the argument. By elementarity,
  $\cM^* = \cN^*$ and $\cM_{\beta+1}^*$ are
  $\varphi$-witnesses. Moreover, as $\cN$ is a minimal
  $\varphi$-witness, by elementarity the same holds for $\cM_\xi^\cU$,
  where $\xi$ is as in the proof of Lemma
  \ref{lem:Pi12omega1itGeneral} the $U$-predecessor of $\beta+1$. But
  $\cM_{\beta+1}^* \lhd \cM_\xi^\cU$, contradicting the minimality
  property for $\cM_\xi^\cU$. Therefore, both sides of the comparison
  do not drop in model.
\end{proof}

We are now going to define a collection of games of length
$\omega^2$ which are generalizations of the game in
\cite[Lemma 3]{MaSt08}. The argument there goes back to ideas in
\cite{Fr71} 
allowing one of the two players in the game to play the theory of a
model with certain properties in addition to the usual moves. In the
proofs of Theorem \ref{TheoremADM1} and Lemma
\ref{LemmaRealsStationary} below we will consider two different
instances of games from this collection where Player I plays a
complete and consistent theory in the language of premice with
additional constant symbols.

Before we give the definition of the games, recall that if
$X \in \ctbleset$ and $\cM$ is an $X$-premouse, then analogously to
the existence of a definable well-order in $L$, there is a uniformly
definable $X$-parametrized family of well-orders the union of whose
ranges is $\cM$ (cf. \cite[Proposition 2.4]{St08b}). More
specifically, we can fix a formula $\theta(\cdot,\cdot,\cdot)$ in the
language of premice $\lpm$ such that for any such $X$ and any
$X$-premouse $\cM$, the following hold:
\begin{enumerate}[$(i)$]
\item for any $x \in \cM$, there is some $\alpha \in \Ord^\cM$ and
  some $r \in X$ such that
  \[\cM \models \theta(\alpha,r,x);\]
\item for all $r \in X$ and $\alpha \in \Ord^\cM$ there is at most one
  $x \in \cM$ such that
  \[\cM \models \theta(\alpha,r,x).\]
\end{enumerate}

Moreover, fix recursive bijections $m$ and $n$ assigning an odd number
$>1$ to each $\langpm$-formula $\varphi$ such that $m$ and $n$ have
disjoint recursive ranges and for every $\varphi$, $m(\varphi)$ and
$n(\varphi)$ are larger than
$\max\{ i \, \colon \, \dot x_i \text{ occurs in } \varphi\}$.

\begin{definition}\label{def:Gvarphipsi}
  Let $\varphi$ and $\psi(x_0,a,b)$ be $\lpm$-formulae and let
  $\mathcal{G}_{\varphi,\psi}$ denote the following game of length
  $\omega^2$ on $\omega$: Fix some enumeration
  $(\phi_i \colon i \in \omega)$ of all $\langpm$-formulae such that
  $\dot x_i$ does not appear in $\phi_j$ if $j \leq i$. Then a typical
  run of $\mathcal{G}_{\varphi,\psi}$ looks as follows:
  \[ \begin{array}{c|ccccccc} \mathrm{I} & x_0 & a & v_0, x_1 & & v_1,
      x_3 & & \hdots \\ \hline \mathrm{II} & & b & & x_2 & & x_4 &
      \hdots \end{array} \]
  \begin{enumerate}
  \item Player I starts by playing some parameter $x_0 \in \BS$;
  \item Players I and II take turns playing natural numbers to
    construct reals $a, b \in \BS$;
  \item Players I and II take turns, respectively playing sequences of
    natural numbers $(v_i, x_{2i+1})$ and $x_{2i+2}$ in $\BS$, for
    $i \in \omega$. We ask that $v_i \in \{0,1\}$.
  \end{enumerate}

  Here $v_i$ will be interpreted as the truth value of the formula
  $\phi_i$ from the enumeration fixed above. This can be thought of as
  Player I either accepting or rejecting the formula $\phi_i$. If so,
  the play determines a complete theory $T$ in the language $\langpm$.

  Player I wins the game $\mathcal{G}_{\varphi,\psi}$ if, and only if,
  \begin{enumerate}
  \item \label{rule:x_1}  $x_1 = a \oplus b$.
  \item \label{rule:reals} For each $i \in \omega$, $T$ contains the
    sentence $\dot x_i \in \BS$ and, moreover, for each
    $j,m \in \omega$, $T$ contains the sentence $\dot x_i(m) = j$ if,
    and only if, $x_i(m) = j$.
  \item \label{rule:dc} For every formula $\phi(x)$ with one free
    variable in the expanded language $\langpm$, and $m(\phi)$ and
    $n(\phi)$ as fixed above, $T$ contains the statements
    \[ \exists x\, \phi(x) \to \exists x\, \exists \alpha \,
      (\phi(x) \land \theta(\alpha, \dot x_{m(\phi)}, x)), \]
    \[ \exists x\, (\phi(x) \land x \in \dot X) \to \phi(\dot
      x_{n(\phi)}). \]
  \item \label{rule:pm} $T$ is a complete, consistent theory such that
    for every countable model $\cM$ of $T$ and every model $\cN^*$
    which is the definable closure of $\{x_i : i < \omega\}$ in
    $\cM \upharpoonright \lpm$, $\cN^*$ is well-founded and if $\cN$
    denotes the transitive collapse of $\cN^*$,
    \begin{enumerate}
    \item $\cN$ is an $X$-premouse, where
      $X = \{x_i \colon i \in \omega\}$,
    \item \label{rule:pmVarphi} $\cN$ is a minimal $\varphi$-witness,
    \item $\cN$ is $\Pi^1_2$-iterable in the sense of \cite[Definition
      1.6]{St95}, and
    \item \label{rule:pmPsi} $\cN \vDash \psi(x_0,a,b)$.
  \end{enumerate}
\end{enumerate}
If Player I plays according to all these rules, he wins the game. In
this case there is a unique premouse $\cN_p$ as in (\ref{rule:pm})
associated to the play $p = (x_0,a \oplus b,v_0, x_1, x_2, \dots)$ of
the game. Otherwise, Player II wins.
\end{definition}

\begin{remark}
  Rule \eqref{rule:dc} in the game $\mathcal{G}_{\varphi,\psi}$
  ensures that if $\cM$ is a model of the theory $T$, then the
  definable closure of $\{x_i:i\in\omega\}$ in
  $\cM \upharpoonright \lpm$ is an elementary substructure of
  $\cM \upharpoonright \lpm$ (by the Tarski-Vaught criterion) by the
  following argument: Suppose $\exists x \phi(x)$ holds in $\cM$. Then
  rule \eqref{rule:dc} ensures that
  $\exists x\, \exists \alpha \, (\phi(x) \land \theta(\alpha, \dot
  x_{m(\phi)}, x))$. Now, the formula
  $\phi(x) \wedge \exists \alpha (\theta(\alpha, \dot x_{m(\phi)}, x)
  \wedge \forall \beta \in \alpha \, \neg \exists y(\phi(y) \land
  \theta(\beta, \dot x_{m(\phi)}, y)))$ uniquely defines a witness $x$
  for $\phi(x)$ (the minimal witness according to the well-order given
  by $\theta(\cdot, \dot x_{m(\phi)}, \cdot)$). Hence, rule
  \eqref{rule:pm} can be followed by Player I by playing an
  appropriate theory $T$, as then the model $\cN$ is uniquely determined
  by it, by Lemma \ref{LemmaNDoesNotDependOnM}.
\end{remark}

To prove Theorem \ref{TheoremADM1}, we first need to show the
following lemma. We thank John Steel for pointing out to us that it
can be proved via a modification of our argument for Theorem
\ref{TheoremADM1}.

\begin{lemma}\label{LemmaRealsStationary}
  Suppose that $\PI^1_2$ games of length $\omega^2$ are
  determined. Then there is a club $\mathcal{C}^* \subset
  \mathcal{P}_{\omega_1}(\mathbb{R})$ such that for all $A \in
  \mathcal{C}^*$, 
  \[ \mathbb{R}\cap M_1(A) = A. \]
\end{lemma}

\begin{proof}
  Assume towards a contradiction that the statement of the lemma
  fails. Thus, there is a stationary set of sets
  $A \in \mathcal{P}_{\omega_1}(\mathbb{R})$ such that
  \[A \subsetneq \mathbb{R}\cap M_1(A).\] Let
  \[ \varphi \equiv \text{``there is a real } y \text{ which is not in
    } \dot X \text{''} \] and
  \begin{gather*}
    \psi(x_0,a,b) \equiv \text{``there is a real definable from } x_0
    \text{ which is not in } \dot X \text{ and } \\  \text{ if } z_0
    \text{ is the least
      real definable from } x_0 \text{ which is not in } \dot X, \\ 
                                                               \;\;\;\; \text{ then
      its } b_0\text{th digit is } a_1\text{''}, 
  \end{gather*}
  where $a = (a_0, a_1, \dots)$ and $b = (b_0, b_1, \dots)$ with
  $a_i,b_i \in \omega$ for all $i \in \omega$. This will only be
  applied in $X$-premice $\cM$ for some $X \in \ctbleset$ with
  $x_0 \in X$ and ``least'' refers to the least real in the well-order
  of elements of $\cM$ definable from $x_0$ which is given by
  $\theta(\cdot, x_0, \cdot)$.
  
  Consider the game $\mathcal{G}_{\varphi,\psi}$, i.e., after Player I
  plays the parameter $x_0$, the only relevant moves are the
  following: Player II plays a natural number $b_0$ asking Player I
  for the $b_0$th digit of the least real definable from $x_0$ which
  is not going to be in $X$ and Player II answers by playing
  $a_1$. Afterwards they continue playing the rest of
  $X = \{ x_0, x_1, \dots \}$ and the theory of a $\varphi$-witness.

  The winning condition in this game $\mathcal{G}_{\varphi,\psi}$ is
  $\Pi^1_2$, so it is determined and we can distinguish the following
  two cases to obtain a contradiction by arguing that no player can
  have a winning strategy.

\medskip
\noindent\paragraph{\bf{Case 1:}}Player I has a winning strategy
$\sigma$ in $\mathcal{G}_{\varphi,\psi}$.

Let $W$ be the transitive collapse of a countable elementary
substructure $Y$ of some large $V_\kappa$ such that $\sigma \in Y$ and
let $\pi$ denote the inverse of the collapse embedding, i.e.,
\[ \pi \colon W \cong Y \prec V_\kappa. \]

Since $\mathbb{R}^W$ is countable, it follows that
$M_1^\sharp(\mathbb{R}^W)$ exists and is $\omega_1$-iterable. Since
the set of $\mathbb{R}^W$ for such elementary substructures $W$ is a
club in $\mathcal{P}_{\omega_1}(\mathbb{R})$, we may assume that
\[\mathbb{R}^W \subsetneq \mathbb{R}\cap M_1(\mathbb{R}^W).\]

The game $\cG_{\varphi,\psi}$ can be defined in $W$. Let
$\bar\sigma \in W$ be such that $\pi(\bar{\sigma}) = \sigma$, i.e.,
$\bar{\sigma} = \sigma \cap W$. By elementarity,
\[ W \vDash \text{``$\bar\sigma$ is a winning strategy for Player I in
    $\mathcal{G}_{\varphi,\psi}$.''} \]

Let $h$ be a well-ordering of $\mathbb{R}^W$ in $V$ of order-type
$\omega$. Consider a play of the game $\mathcal{G}_{\varphi,\psi}$ in
$V$ in which Player II plays some $b\in\BS$ and $x_2, x_4, \dots$
according to $h$ and Player I plays according to the winning strategy
$\sigma$. Every proper initial segment of the play is in the domain of
$\bar\sigma$. It follows that the real part $(x_0, x_1, x_2, \dots)$
of the play, say $p$, enumerates $\mathbb{R}^W$. Furthermore, $p$ is
consistent with $\sigma$, whereby $p$ is won by Player I. This means
that $p$ determines a $\Pi^1_2$-iterable $\mathbb{R}^W$-premouse
$\cN_p$ which is a minimal $\varphi$-witness.

Let $\delta_{\bR^W}$ denote the Woodin cardinal in $M_1(\bR^W)$. Since
we chose $W$ so that
$\mathbb{R}^W \subsetneq \mathbb{R}\cap
M_1(\mathbb{R}^W)|\delta_{\mathbb{R}^W}$ and since satisfying
$\varphi$ for the $\bR^W$-premice $\cN_p$ and
$M_1(\bR^W)|\delta_{\bR^W}$ means having a real which is not in
$\bR^W$, we have that $M_1(\mathbb{R}^W)|\delta_{\bR^W}$ is an
$\omega_1$-iterable $\varphi$-witness. Thus, Lemma
\ref{lem:Pi12omega1itGeneral} implies that $\cN_p$ is
$\omega_1$-iterable as well.

Let $x_0 \in \BS$ be the first move given by $\sigma$ (so $x_0$ is
also the first move given by $\bar\sigma$). Moreover, let
$a_0 \in \omega$ be the first move of Player I after $x_0$ given by
$\sigma$ (and $\bar \sigma$). Let $\tau$ be the real defined by
\[\tau(n) = \bar\sigma(x_0, a_0, n),\] for all possible moves $n \in \omega$
of Player II for $b_0$. We claim that $\tau$ is the least real in
$\cN_p$ not in $\mathbb{R}^W$ which is definable from $x_0$. This will
be a contradiction, since $\tau \in W$ as $\bar\sigma \in W$ and
$x_0\in W$.

Let $\tau^\prime$ be the least real in $\cN_p$ not in $\mathbb{R}^W$
which is definable from $x_0$. Assume that $\tau^\prime \neq \tau$ and
choose some $n_0 \in\omega$ such that
$\tau(n_0) \neq \tau^\prime(n_0)$. Let $q$ be the play of the game
$\mathcal{G}_{\varphi,\psi}$ in which Player I plays according to
$\bar\sigma$ and Player II plays some $b \in \BS$ with first digit
$n_0$ and then $h$ as above. As Player I plays according to
$\bar\sigma$ and hence according to $\sigma$, this is a winning play
for Player I. Let $\cN_q$ be the corresponding model. In particular,
$\cN_q \vDash \psi(x_0,a,b)$, i.e. the least real in $\cN_q$ not in
$\mathbb{R}^W$ which is definable from $x_0$ has
$a_1 = \bar\sigma(x_0, a_0, n_0)$ as $n_0$th digit.

By the rules of the game, $\cN_q$ is a $\Pi^1_2$-iterable
$\bR^W$-premouse which is a minimal $\varphi$-witness. Hence, Lemma
\ref{lem:Pi12omega1itGeneral} yields (as in the case of $\cN_p$) that
$\cN_q$ is in fact $\omega_1$-iterable. So Lemma
\ref{lem:CommonIterate} implies that $\cN_p$ and $\cN_q$ coiterate to
a common model and there is no drop in model on the main branch
through the trees on both sides of the coiteration. By definition,
both $\cN_p$ and $\cN_q$ are pointwise definable from
$\bR^W$. Therefore it is easy to see that in fact $\cN_p = \cN_q$. In
particular, $\cN_p$ and $\cN_q$ have the same least real $\tau^\prime$
definable from $x_0$ which is different from all reals in $\bR^W$ and
by choice of $q$,
$\tau^\prime(n_0) = \bar\sigma(x_0, a_0, n_0) = \tau(n_0)$, which is
the desired contradiction.

\medskip
\noindent\paragraph{\bf{Case 2:}}
Player II has a winning strategy $\sigma$ in
$\mathcal{G}_{\varphi,\psi}$.

Let $W$ be the transitive collapse of a countable elementary
substructure $Y$ of some large $V_\kappa$ such that $\sigma \in Y$ and
let $\pi$ denote the inverse of the collapse embedding. Moreover, let
$\bar \sigma \in W$ be such that $\pi(\bar \sigma) = \sigma$, i.e.,
$\bar \sigma = \sigma \cap W$. Since $\mathbb{R}^W$ is countable, it
follows that $M_1^\sharp(\mathbb{R}^W)$ exists and is
$\omega_1$-iterable in $V$. As before, by our hypothesis we may assume
that
\[\mathbb{R}^W \subsetneq \mathbb{R}\cap M_1(\mathbb{R}^W).\]

Let $\cQ = M_1(\bR^W)|\alpha$, where $\alpha$ is least such that
$\cQ \vDash \zf + $ ``there are no Woodin cardinals'' and $\cQ$
contains a real which is not in $\mathbb{R}^W$. Let $\cN^{*,\cQ}$ be
the definable closure of $\bR^W$ in $\cQ$ and $\cN^\cQ$ the transitive
collapse of $\cN^{*,\cQ}$. Then $\cN^\cQ \prec \cQ$. Thus, there is
some real $z$ in $\cN^\cQ$ which is not in $\bR^W$ such that $z$ is
definable in $\cN^\cQ$ from some real $x_0 \in \bR^W$. We shall ask
Player I to begin every play of the game $\mathcal{G}_{\varphi,\psi}$
by playing this real $x_0$. Assume without loss of generality that $z$
is the least real in $\cN^\cQ\setminus \bR^W$ definable from $x_0$
according to the well-order defined by $\theta(\cdot, x_0, \cdot)$.

Consider the play $p$ in $\mathcal{G}_{\varphi,\psi}$ in which Player
II plays according to $\bar\sigma$ (and hence according to the winning
strategy $\sigma$) and Player I plays:
\begin{enumerate}
\item $x_0 \in \bR^W$, in the first round,
\item $a_1 = z(b_0)$, in response to Player II playing $b_0 \in
  \omega$ according to $\sigma$,
\item other, arbitrary, natural numbers $a_0, a_2, a_3, \dots$,
\item some enumeration $h$ of $\mathbb{R}^{W}$ in order-type $\omega$
  with $h\in V$ as in Case 1, together with the theory of $\cQ$ in the
  language $\langpm$, where the constants $\dot x_i$ are interpreted
  by the reals $x_i \in \bR^W$ according to $p$, satisfying rules
  \eqref{rule:x_1}, \eqref{rule:reals}, and \eqref{rule:dc} of the
  game $\mathcal{G}_{\varphi,\psi}$.
\end{enumerate} 
Arguing as before, one shows that the reals in played in $p$ enumerate
$\bR^W$. It follows that the model $\cN^{\cQ}$ witnesses that $p$ is
a winning play for Player I, which contradicts the fact that $\sigma$
is a winning strategy for Player II. This proves the lemma.
\end{proof}

\begin{remark}
  It is also possible to prove Lemma \ref{LemmaRealsStationary} with
  the following variant of the argument we gave above. Instead of
  playing a code for a theory via $v_i \in \{0,1\}$ for
  $i \in \omega$, we could ask Player I to play a (fine structural)
  code for a premouse $\cN$ projecting to $\omega$ digit by digit via
  $v_i \in \omega$ for $i \in \omega$. In addition, we can let Player
  I play $w_i \in \omega$ together with $v_i$ for each $i \in \omega$
  to play another real $w$ digit by digit. Then we say Player I wins
  if, and only if, the premouse $\cN$ he codes satisfies (a)-(c) in
  Definition \ref{def:Gvarphipsi} where we no longer require $\zf$ in
  the definition of $\varphi$-witness, but ask that $w$ is the minimal
  (in the natural order on formulae and ordinal parameters) real
  definable from $x_0$ over $\cN$ which is not in
  $X = \{x_i \colon i \in \omega\}$ and that no proper initial segment
  of $\cN$ satisfies this property. Then a similar argument as in
  Lemma \ref{lem:Pi12omega1itGeneral} shows that if there is an
  $\omega_1$-iterable model with this property, $\cN$ is in fact
  $\omega_1$-iterable as well. Now a similar argument as in the proof
  of Lemma \ref{LemmaRealsStationary} above shows that this game
  works. Moreover, the same idea can be used to phrase the proof of
  Theorem \ref{TheoremADM1} below differently.
\end{remark}

Finally, we are ready to prove Theorem \ref{TheoremADM1}.

\begin{proof}[Proof of Theorem \ref{TheoremADM1}]
Assume towards a contradiction that we have
\[ M_1(B)|\delta_B \not \models \ad\] for a stationary set of sets
$B \in \mathcal{P}_{\omega_1}(\mathbb{R})$, where $\delta_B$ denotes
the Woodin cardinal in $M_1(B)$. Let
  \[ \varphi \equiv \text{``} \dot X = \bR + \neg \ad \text{''} \] and
  \begin{gather*}
    \psi(x_0,a,b) \equiv \text{``there is a non-determined set of
      reals definable from }x_0 \text{ and } \\ \text{if } Z \text{ is
      the least such set in the well-order relative to } \dot X, \\
    \;\;\;\; \text{ then } a \oplus b \in Z \text{''}.
  \end{gather*}
  This will only be applied in $X$-premice $\cM$ for some
  $X \in \ctbleset$ with $x_0 \in X$ and as in the proof of Lemma
  \ref{LemmaRealsStationary} ``least'' refers to the least set in the
  well-order of elements of $\cM$ definable from $x_0$ which is given
  by $\theta(\cdot, x_0, \cdot)$.

  In this case, the game $\cG_{\varphi,\psi}$ is a variant of the
  Kechris-Solovay game in \cite{KS85} (see also the game in
  \cite[Lemma 2.3]{MSW}) adapted as a model game. The winning
  condition is $\Pi^1_2$, so the game $\cG_{\varphi,\psi}$ is
  determined. We will obtain a contradiction by arguing that no player
  can have a winning strategy.

\medskip
\noindent\paragraph{\bf{Case 1:}}Player I has a winning strategy
$\sigma$ in $\mathcal{G}_{\varphi,\psi}$.

Let $W$ be the transitive collapse of a countable elementary
substructure $Y$ of some large $V_\kappa$ with $\sigma \in Y$ and let
$\pi$ denote the inverse of the collapse embedding, i.e.
\[ \pi \colon W \cong Y \prec V_\kappa. \]
 
Since $\mathbb{R}^W = \bR \cap W$ is countable, it follows that
$M_1^\sharp(\mathbb{R}^W)$ exists and is $\omega_1$-iterable (in
$V$). By Lemma \ref{LemmaRealsStationary}, the set of $\bR^W$ for $W$
as above with the additional property that
$M_1(\bR^W) \cap \bR = \bR^W$ is club in
$\mathcal{P}_{\omega_1}(\mathbb{R})$. By assumption, we may thus
choose $W$ so that $M_1(\bR^W) \cap \bR = \bR^W$ and in addition
\[ M_1(\mathbb{R}^W)|\delta_{\mathbb{R}^W}\not\models\ad. \]

Note that the game $\cG_{\varphi,\psi}$ can be defined in $W$ and let
$\bar\sigma \in W$ be such that $\pi(\bar{\sigma}) = \sigma$, i.e.,
$\bar{\sigma} = \sigma \cap W$. By elementarity,
\[ W \models \text{``$\bar\sigma$ is a winning strategy for Player I
    in $\mathcal{G}_{\varphi,\psi}$.''} \]

Let $h$ be a well-ordering of $\mathbb{R}^W$ in $V$ of order-type
$\omega$. Clearly, every proper initial segment of $h$ is in
$W$. Consider a play of the game $\mathcal{G}_{\varphi,\psi}$ in $V$
in which Player II plays some arbitrary $b^* \in \BS$ and
$x_2, x_4, \dots$ according to $h$ and Player I plays according to the
winning strategy $\sigma$. Every proper initial segment of the play is
in the domain of $\bar\sigma$. It follows that the real part
$(x_0, x_1, x_2, \dots)$ of the play, say $p$, enumerates
$\mathbb{R}^W$. Futhermore, $p$ is consistent with $\sigma$, whereby
$p$ is won by Player I. This means that $p$ determines a
$\Pi^1_2$-iterable $\mathbb{R}^W$-premouse $\cN_p$ which is a minimal
$\varphi$-witness. In particular, $\cN_p \cap \bR = \bR^W$. Since both
$\cN_p$ and $M_1(\mathbb{R}^W)|\delta_{\bR^W}$ are
$\mathbb{R}^W$-premice and $M_1(\mathbb{R}^W)|\delta_{\bR^W}$ is a
$\varphi$-witness, $\cN_p$ is $\omega_1$-iterable by Lemma
\ref{lem:Pi12omega1itGeneral}.

Let $x_0 \in \BS$ be the first move given by $\sigma$ (so $x_0$ is
also the first move given by $\bar\sigma$). Let $Z = Z(x_0,\cN_p)$
denote the least non-determined set of reals in $\cN_p$ which is
definable from $x_0$. This exists since
$\cN_p \vDash \psi(x_0,a^*,b^*)$, where $a^*$ is the sequence of
natural numbers Player I plays after $x_0$ in response to $b^*$
according to $\sigma$. There is a natural strategy $\tau$ for Player I
in $G(Z)$---the Gale-Stewart game with winning condition $Z$ played in
$\cN_p$---which is induced by $\bar\sigma$. Let $\tau$ be the unique
strategy such that for $a,b \in \BS \cap \cN_p$,
\[a = \tau(b) \text{ if, and only if, } (x_0,a) = \bar\sigma(b).\]
Note that $\tau \in W$ as $\bar\sigma, x_0 \in W$ and, since the reals
of $\cN_p$ are those of $W$, we also have $\tau \in \cN_p$.

We claim that $\tau$ is a winning strategy for Player I (in the game
$G(Z)$ in $\cN_p$), which will contradict the fact that the set $Z$ is
non-determined in $\cN_p$. Let $a \oplus b\in \mathbb{R}^W$ be a play
by $\tau$. Let $q$ be the play of the game
$\mathcal{G}_{\varphi,\psi}$ in which Player I plays according to
$\bar\sigma$ and Player II plays $b$ and then $h$ as above. As Player
I plays according to $\bar\sigma$ and hence according to $\sigma$,
this is a winning play for Player I. Let $\cN_q$ be the corresponding
model. In particular, $\cN_q \vDash \psi(x_0,a,b)$, i.e.
$a \oplus b \in Z(x_0,\cN_q)$, where $Z(x_0,\cN_q)$ denotes the least
non-determined set of reals in $\cN_q$ which is definable from $x_0$.

By the rules of the game, $\cN_q$ is a $\Pi^1_2$-iterable
$\bR^W$-premouse which is a minimal $\varphi$-witness. Hence Lemma
\ref{lem:Pi12omega1itGeneral} yields that $\cN_q$ is in fact
$\omega_1$-iterable. So we can apply Lemma \ref{lem:CommonIterate} to
$\cN_q$ and $\cN_p$. In fact, $\cN_p = \cN_q$ as both are pointwise
definable from $\bR^W$. Therefore, $Z = Z(x_0,\cN_p) = Z(x_0,\cN_q)$
and $a \oplus b \in Z$. Hence, $\tau$ is a winning strategy for Player
I, contrary to the fact that $Z$ is non-determined in $\cN_p$.

\medskip
\noindent\paragraph{\bf{Case 2:}} Player II has a winning strategy
$\sigma$ in $\mathcal{G}_{\varphi,\psi}$.

As before, let $W$ be the transitive collapse of a countable
elementary substructure $Y$ of some large $V_\kappa$ with
$\sigma \in Y$ and let $\pi$ denote the inverse of the collapse
embedding, i.e.
\[ \pi \colon W \cong Y \prec V_\kappa. \]
Then $M_1^\sharp(\mathbb{R}^W)$ exists and is $\omega_1$-iterable in $V$.
As before, we may choose $W$ so that
\[M_1(\mathbb{R}^W) \cap \bR = \bR^W \text{ and }
  M_1(\mathbb{R}^W)|\delta_{\mathbb{R}^W}\not\models\ad.\] Let
$\bar\sigma = \sigma\cap W$, so that $\bar\sigma \in W$ and
\[ W\models \text{``$\bar\sigma$ is a winning strategy for Player II in
    $\mathcal{G}_{\varphi,\psi}$.''} \]

Let $\cQ = M_1(\bR^W)|\alpha$, where $\alpha$ is least such that there
are no Woodin cardinals in $\cQ$ and $\cQ \models \zf+\neg\ad$. Let
$\cN^{*,\cQ}$ be the definable closure of $\bR^W$ in $\cQ$ and
$\cN^\cQ$ the transitive collapse of $\cN^{*,\cQ}$. Then
$\cN^\cQ \prec \cQ$ and $\cN^\cQ$ is $\omega_1$-iterable because it is
elementary embedded in the $\omega_1$-iterable premouse
$\cQ$. Moreover, there is some non-determined set in $\cN^\cQ$
definable from some real $x_0 \in \bR^W$. We shall ask Player I to
play this real $x_0$ followed by some real $a$ and some enumeration
$h \in V$ of $\bR^W$ in order-type $\omega$ together with the theory
of $\cQ$ (of course, organized in such a way that he satisfies rules
\eqref{rule:x_1}, \eqref{rule:reals}, and \eqref{rule:dc} of the
game). Let $Z(x_0,\cN^\cQ)$ be as before.

Since $\bar\sigma$ is a winning strategy for Player II in the game
$\mathcal{G}_{\varphi,\psi}$ in $W$, in particular $\bar\sigma$ wins
against all plays in which Player I begins by playing $x_0$. As
before, there is a natural strategy $\tau \in \cN^\cQ$ for Player II
for the Gale-Stewart game inside $\cN^\cQ$ with payoff set
$Z(x_0,\cN^\cQ)$, namely, the unique strategy such that
\[b = \tau(a) \text{ if, and only if, } b = \bar\sigma(x_0, a).\] A
similar argument as in Case 1, using that Player I cannot lose a play
$p$ as above because of having played the wrong theor, but only
because $a\oplus b \notin Z(x_0,\cN_p)$, gives that $\tau$ is a
winning strategy for Player II in $\cN^\cQ$. This finishes the proof
of Theorem \ref{TheoremADM1}.
\end{proof}

\section{Dependent Choices, Scales, and Mouse
  Capturing}\label{SectionDC}

By Theorem \ref{TheoremADM1} we obtain a countable set of reals $A$
such that $M_1(A)$ is an $A$-premouse constructed over its reals and
$M_1(A) \vDash \zf + \ad$ from the assumption that all $\PI^1_2$ games
of length $\omega^2$ are determined. We aim to show that there is a
model with $\omega+1$ Woodin cardinals from this hypothesis. Before we
do that in the next section, we first show some structural properties
of this model.

First, we have that in fact $M_1(A) \vDash \dc$ by the following
theorem, which is a special case of \cite[Theorem 1.1]{Mu} (building
on \cite{St08b} and \cite{Ke84}).

\begin{theorem}\label{thm:M1DC}
  Let $A \in \ctbleset$ be a countable set of reals such that
  $M_1(A) \cap \bR = A$ and $M_1(A) \vDash \zf + \ad$. Then
  $M_1(A) \vDash \dc$.
\end{theorem}

In what follows we argue that, in $M_1(A)$, $\Sigma^2_1$ has the scale
property and $\Theta = \theta_0$, i.e. the Solovay sequence is
trivial. Assuming $\ad$, recall that
\[ \Theta = \sup\{\beta : \text{there is a surjection } f \colon \bR
  \rightarrow \beta \},\] and
\[\theta_0 = \sup\{\beta : \text{there is an } \od \text{ surjection
  } f \colon \bR \rightarrow \beta\}.\]

These properties of $M_1(A)$ have proofs similar to those for
$L(\bR)$. E.g., using \cite{St08b} the scale analysis of $L(\bR)$ from
\cite{St08a} can be done inside $M_1(A)$ and yields that
$\Sigma_1^{M_1(A)}$ has the scale property. Moreover, as in $L(\bR)$,
it is easy to see that $(\Sigma^2_1)^{M_1(A)} =
\Sigma_1^{M_1(A)}$. Similarly, $(\Theta = \theta_0)^{M_1(A)}$. In
fact, by generalizing the arguments used for $L(\bR)$, we can also get
that $\ad^+$ holds in $M_1(A)$ but we will not need that. We summarize
this in the next theorem.

\begin{theorem}\label{thm:M1ScalesTheta}
  Let $A \in \ctbleset$ be a countable set of reals such that
  $M_1(A) \cap \bR = A$ and $M_1(A) \vDash \zf + \ad$. Then
  $M_1(A) \vDash \text{``}\Sigma^2_1\text{ has the scale property''} +
  \Theta = \theta_0$.
\end{theorem}

Finally, we also have that $M_1(A)$ satisfies \emph{Mouse Capturing}
($\mc$), i.e., that for any two countable transitive sets $x$ and $y$
such that $x \subseteq y$ and $x \in \od_{y \cup \{y\}}$, $x$ is
contained in an $\omega_1$-iterable $y$-premouse. This follows from
\cite[Theorem 1.5]{St16} (due to Woodin).

\begin{theorem}\label{thm:M1MC}
  Let $A \in \ctbleset$ be a countable set of reals such that
  $M_1(A) \cap \bR = A$ and $M_1(A) \vDash \zf + \ad$. Then
  $M_1(A) \vDash \mc$.
\end{theorem}

\section{$\omega+1$ Woodin cardinals} 
\label{SectOmega+1WdnsFromAD}

In this section we will use the results from the previous sections to
construct a premouse with $\omega +1$ Woodin cardinals. More
precisely, we prove the following theorem:

\begin{theorem}\label{thm:omega+1Wdns}
  Suppose there is some $A \in \ctbleset$ such that 
  \begin{enumerate}
  \item $M_1^\sharp(A)$ exists,
  \item $M_1(A) \cap \bR = A$, and
  \item $M_1(A) \vDash \ad$.
  \end{enumerate}
  Then there is an active premouse with
  $\omega+1$ Woodin cardinals.
\end{theorem}

Using Section \ref{SectionDC}, for the rest of this section, we fix a
countable set of reals $A$ such that $M_1(A)$ is an
$\omega_1$-iterable $A$-premouse with $M_1(A) \cap \bR = A$ and
\[ M_1(A) \vDash \zf + \dc + \ad + \text{``}\Sigma^2_1\text{ has the
    scale property''} + \Theta = \theta_0 + \mc. \]

The rest of this section is devoted to the proof of Theorem
\ref{thm:omega+1Wdns}. Most of the proof in this section closely
follows ideas from Section 3 in the unpublished notes \cite{St}. See
also Sections 6.5 and 6.6 in \cite{StW16} for a similar argument
applied to $L(\bR)$ or \cite{SaSt15}. We start by introducing some
notation, generalizing ideas from \cite[Section 3]{StW16} and
\cite{SchlTr} to our context. Suppose that $A$ is as in the statement
of Theorem \ref{thm:omega+1Wdns} and work inside $V = M_1(A)$.

The proof of Theorem \ref{thm:omega+1Wdns} can be split in several
parts. After we recall a useful standard fact, we define suitable
premice. From these we will, by pseudo-comparison and
pseudo-genericity iteration, obtain models which we can use in a
Prikry-like forcing to construct a model with $\omega$ Woodin
cardinals. Then we argue that this model can be rearranged into a
premouse on top of which we can perform a $\cP$-construction to add
one more Woodin cardinal.
 
The following standard lemma will be useful later on:

\begin{lemma} \label{lem:propL[T,a]} There is a $\Sigma^2_1$ scale
  $\vec \phi$ on a $\Sigma^2_1$ set which is universal for
  $\SIGMA^2_1$ such that, letting $T$ be the tree obtained from
  $\vec \phi$, we have for any countable transitive set $a$,
 \begin{align*}
    \mathcal{P}(a)\cap L[T,a]
                                            &= \mathcal{P}(a)\cap
                                              \od_{a \cup\{a\}} \\
                                            &= \{b\in H(\omega_1): b
                                              \text{ belongs to an
                                              $\omega_1$-iterable
                                             $a$-premouse}\}.
 \end{align*}
\end{lemma}

\begin{proof}
  The second equality easily follows from $\mc$. For the first
  equality, let $U\subset\bR^2$ be any $\Sigma^2_1$ set that is
  universal for $\SIGMA^2_1$. $\Sigma^2_1$ has the scale property, so
  let $T$ be a tree on $\omega\times\omega\times\delta^2_1$ obtained
  from a $\Sigma^2_1$-scale on $U$ (thus $T$ projects to
  $U$). 

  Now, suppose $b \in\mathcal{P}(a)\cap \od_{a \cup\{a\}}$. Let $z$ be
  a real coding $a$. Then the set $B$ of all $\od_{a \cup\{a\}}$
  subsets of $a$ (coded as a real relative to $z$) is
  $\Sigma^2_1(z)$. By the Mansfield-Solovay Theorem (see for example
  \cite[Theorem 11.1]{KM08}), either $B \subseteq L[T,z]$, or $B$ contains a perfect
  subset. But since $B$ is countable, it is thin. Therefore $b_z$, the
  real coding $b$ relative to $z$, is in $L[T,z]$ and hence
  $b \in L[T,z]$. Since this holds for all $z$ which are
  $\col(\omega,a)$-generic over $L[T,a]$, it follows that
  $b \in L[T,a]$.
  
  Conversely, every real in $L[T,a]$ is definable from $a$ and ordinal
  parameters. This is because the reals of $L[T,a]$ do not depend on
  the choice of the universal set $U$ nor on the scale on $U$ (this
  follows e.g., from \cite[Exercise 8G.29]{Mo09}, see also
  \cite{HK81}).
\end{proof}

Fix a $\Sigma^2_1$ set $U$ which is universal for $\SIGMA^2_1$ in
$V = M_1(A)$ and a tree $T$ as above for the rest of this section.

\subsection{Suitable premice}

We begin by isolating a class of models suitable for our purposes.

\begin{definition}
Suppose $b$ is a countable transitive set. We write
\[Lp(b) = \bigcup\big\{M: \text{$M$ is a sound $\omega_1$-iterable
    $b$-premouse such that } \rho_\omega(M) = b \big\}.\] Moreover, we
inductively define $Lp^{1}(b) = Lp(b)$,
\[Lp^{n+1}(b) = Lp(Lp^{n}(b)),\] and
\[Lp^{\omega}(b) = \bigcup_{n<\omega} Lp^{n}(b).\]
\end{definition}

\begin{remark}
  Recall that we are working inside $M_1(A)$, which is a model of
  $\ad$, so the club filter is an ultrafilter on $\omega_1$. This can
  be used to show that $\omega_1$-iterability already implies
  $\omega_1+1$-iterability (see e.g. \cite[Lemma 7.11]{St10} for
  details), so that any two $\omega_1$-iterable $b$-premice $M$ and
  $N$ as in the definition of $Lp(b)$ can be successfully compared and
  line up, i.e $M \unlhd N$ or $N \unlhd M$. Therefore, $Lp(b)$ is a
  well-defined premouse.
\end{remark}

In the definitions below, let $a$ be an arbitrary countable transitive
set. 

\begin{definition}
  We say that an $a$-premouse $M$ is \emph{suitable} if, and only if,
  there is an ordinal $\delta$ such that
\begin{enumerate}
\item $M$ is a model of $\zfc$ - ``Replacement'' and
  $M \cap \Ord = \sup_{n<\omega} (\delta^{+n})^M$,
\item $\delta$ is the unique Woodin cardinal in $M$, and
\item $M$ is \emph{full}, i.e. for every cutpoint\footnote{In the case
    where $\eta$ is not a cutpoint, we refer to the *-transformation
    in \cite[Section 7]{St08}.} $\eta$ in $M$, $Lp(M|\eta) \unlhd M$.
\end{enumerate}
\end{definition}

If $M$ is suitable, we denote its Woodin cardinal by $\delta_M$.

\begin{lemma} \label{lem:absolSuitableLxT} Let $M$ be a countable
  $a$-premouse and $x_M$ a real coding $M$. Then for any real
  $z \geq_T x_M$, the statement ``$M$ is suitable'' is absolute
  between $V$ and $L[T,z]$.
\end{lemma}

\begin{proof}
  Suppose not, say $M$ is suitable in $L[T,z]$ but there is some
  $\eta < M \cap \Ord$ and a sound $\omega_1$-iterable
  $M|\eta$-premouse $N$ in $V$ with $\rho_\omega(N) = M|\eta$ such
  that $N \ntrianglelefteq M$. This statement is $\SIGMA^2_1$ and
  hence such a counterexample would also exist in $L[T,z]$. By the
  same argument, if we suppose that $M$ is suitable in $V$, every such
  $\omega_1$-iterable $M|\eta$-premouse in $L[T,z]$ is also
  $\omega_1$-iterable in $V$. Hence the statement ``$M$ is suitable''
  is absolute between $V$ and $L[T,z]$.
\end{proof}

\begin{definition}
  Let $\cT$ be a normal iteration tree on a suitable $a$-premouse $M$
  of length $< \omega_1^V$. Then we say that $\cT$ is \emph{correctly
    guided} if, and only if, for every limit ordinal $\lambda < \lh(\cT)$, if $b$
  is the branch choosen through $\cT \upharpoonright \lambda$ in
  $\cT$, then $\cQ(b, \cT \upharpoonright \lambda)$ exists and
  $\cQ(b, \cT \upharpoonright \lambda) \unlhd Lp(\cM(\cT
  \upharpoonright \lambda))$.
\end{definition}

\begin{definition}
  Let $\cT$ be a normal iteration tree on a suitable $a$-premouse $M$
  of length $< \omega_1^V$. Then we say that $\cT$ is \emph{short} if,
  and only if, $\cT$ is correctly guided and if $\cT$ has limit
  length, then $\cQ(\cT)$ exists, and $\cQ(\cT) \unlhd
  Lp(\cM(\cT))$. If $\cT$ is correctly guided but not short, then it
  is said to be \emph{maximal}.
\end{definition}

As in \cite{SchlTr}, we define the notion of being suitability-strict
in order to make the proofs of Lemmas
\ref{lem:PseudoComparisonInL[x,T]} and \ref{lem:PseudoGenItInL[x,T]}
below work.

\begin{definition}
  Let $M$ be a suitable $a$-premouse and let $\cT$ be a normal
  iteration tree on $M$ of length $< \omega_1^V$. Then we say that
  $\cT$ is \emph{suitability-strict} if, and only if, for all
  $\alpha < \lh(\cT)$,
  \begin{enumerate}[$(i)$]
  \item if $[0, \alpha]_T$ does not drop then $\cM_\alpha^\cT$ is
    suitable, and
  \item if $[0, \alpha]_T$ drops then no $\cR \unlhd \cM_\alpha^\cT$
    is suitable.
  \end{enumerate}
\end{definition}

\begin{definition}
  Let $M$ be a suitable $a$-premouse. Then we say that $M$ is
  \emph{short tree iterable} if, and only if, whenever $\cT$ is a
  short tree on $M$ of length $< \omega_1^V$,
  \begin{enumerate}[$(i)$]
  \item $\cT$ is suitability-strict,
  \item if $\cT$ has a last model, then every putative iteration tree
    $\cU$ extending $\cT$ such that $\lh(\cU) = \lh(\cT) +1$ has a
    well-founded last model, and
  \item if $\cT$ has limit length, then there exists a cofinal
    well-founded branch $b$ through $\cT$ such that
    $\cQ(b,\cT) = \cQ(\cT)$.
  \end{enumerate}
\end{definition}

\begin{definition}
  Suppose $M$ is a suitable $a$-premouse. We say $R$ is a
  \emph{pseudo-normal iterate} of $M$ if, and only if, $R$ is suitable
  and there is a normal iteration tree $\cT$ on $M$ such that either
  $\cT$ has successor length and $R$ is the last model of $\cT$ or
  $\cT$ is maximal and $R = Lp^{\omega}(\cM(\cT))$.
\end{definition}

As usual, this notion can easily be generalized to stacks of normal
trees, but we omit the technical details. The interested reader can
find them in a different setting for example in \cite{StW16},
\cite{Sa13}, or \cite{MuSa}. The following lemmas are the analogues of
Theorems 3.14 and 3.16 in \cite{StW16}. The proofs are similar to the
ones in \cite{StW16} and \cite{SchlTr} and use absoluteness as in the
proof of Lemma \ref{lem:absolSuitableLxT}; we omit further details.

\begin{lemma}[Pseudo-comparison]\label{lem:PseudoComparisonInL[x,T]}
  Suppose $M$ and $N$ are countable short tree iterable, suitable
  $a$-premice. Then, they have a common pseudo-normal iterate $R$ such
  that $R \in L[T,z]$, where $z$ is a real coding $M$ and
  $N$. Moreover,
  $\delta_{R} \leq (\max\{\delta_M,\delta_N\}^+)^{L[T,z]} =
  \omega_1^{L[T,z]}$. 
\end{lemma}

\begin{lemma}[Pseudo-genericity iteration]\label{lem:PseudoGenItInL[x,T]}
  Let $M$ be a countable, short tree iterable, suitable
  $a$-premouse. Then for a cone of reals $z$, $M$ is countable in
  $L[T,z]$, and there is a non-dropping pseudo-normal iterate $R$ of
  $M$ in $L[T,z]$ such that $z$ is generic over $R$ for Woodin's
  extender algebra at $\delta_R$, and
  $\delta_{R} = \omega_1^{L[T,z]}$.
\end{lemma}

\subsection{The models $\cR^x_a$}
As before, let $a$ be an arbitrary countable transitive set. Let $x$
be a real such that $a \in L[x]$.  We consider the simultaneous
pseudo-comparison of all short tree iterable, suitable $a$-premice
coded by some real $z \leq_T x$. We carry out this pseudo-comparison
while at the same time performing a pseudo-genericity iteration making
every $z \leq_T x$ generic over the common part of the final
model. Note that there are only countably many reals $z \leq_T x$ for
every fixed real $x$ and let $\cR^- = \cR_a^{x,-}$ denote the
resulting model. I.e., either $\cR^-$ is the common last model of the
iteration trees obtained by the process described above, or all of
these trees are maximal and $\cR^-$ is the common part model of one
(and hence all) of these trees. Moreover, let
$\cR = \cR_a^x = Lp(\cR_a^{x,-})$ and
$\delta = \cR_a^{x,-} \cap \Ord$.

\begin{lemma} \label{lem:propRax} We have the following properties.
  \begin{enumerate}
  \item As an $\cR^{-}$-premouse, no level of $\cR$ projects across
    $\delta$, \label{eq:RNotProjAcross} 
  \item $\delta_{\cR} = \delta$ is a Woodin cardinal in
    $\cR$, \label{eq:RdeltaWoodin}
  \item
    $\mathcal{P}(\delta)\cap \cR = \mathcal{P}(\delta) \cap \od_{\cR^{-}
      \cup \{\cR^{-}\}},$ \label{eq:ODdeltainR}
  \item
    $\mathcal{P}(a)\cap \cR = \mathcal{P}(a) \cap \od_{a \cup \{a\}},$
    \label{eq:ODinR} and
  \item $\omega_1^{L[T,x]} = \delta$. \label{eq:omega1delta}
  \end{enumerate}
\end{lemma}

\begin{proof} We carry out the pseudo-comparisons and
  pseudo-genericity iterations to obtain $\cR^-$ within $L[T,x]$ in
  the sense of Lemmas \ref{lem:PseudoComparisonInL[x,T]} and
  \ref{lem:PseudoGenItInL[x,T]}.

  \begin{claim}\label{cl:ComparisonMaximal}
    The pseudo-comparisons reach a limit stage in which all of the
    iteration trees are maximal.
  \end{claim}
  \begin{proof}
    If the pseudo-comparisons reach a limit stage in which one
    iteration tree $\cT$ is maximal, this already implies that all
    iteration trees are maximal since they agree on their common part
    model and thus a short iteration tree $\cU$ would provide a
    $\cQ$-structure $\cQ(\cU) \unlhd Lp(\cM(\cU)) = Lp(\cM(\cT))$ for
    $\cT$, contradicting the maximality of $\cT$.

    Therefore, we can suppose toward a contradiction that all
    iteration trees occuring in the pseudo-comparisons are short. Then
    the pseudo-comparisons are in fact comparisons and they end
    successfully using the short tree iteration strategies. They give
    rise to a last model $\cR^*$ such that every $z\leq_T x$ is
    generic over $\cR^*$ for Woodin's extender algebra. Moreover, the
    main branches through all iteration trees in the comparisons are
    non-dropping and we have elementary iteration embeddings
    \[j_\cN \colon \cN \to \cR^*\] for each short tree iterable,
    suitable $a$-premouse $\cN$ coded by some real $z\leq_T x$.
    In particular, $\cR^*$ is suitable, witnessed by a Woodin cardinal
    $\delta^* = \delta_{\cR^*}$.

    The proof of the comparison lemma (cf. e.g., the claim in the
    proof of \cite[Theorem 3.11]{St10}) shows that, if a coiteration
    terminates successfully, the comparison process lasts at most
    countably many steps in $L[T,x]$, and so $\cR^*$ is countable in
    $L[T,x]$. By construction, $x$ is generic over $\cR^*$ for
    Woodin's extender algebra at $\delta^*$, so we shall write
    $\cR^*[x]$ for the corresponding generic extension.

    \begin{subclaim}\label{SubclaimR*[x]}
      $\mathbb{R}\cap L[T,x] \subseteq \mathbb{R}\cap \cR^*[x].$
    \end{subclaim}
    \begin{proof}
      Recall that Woodin's extender algebra at $\delta^*$ has the
      $\delta^*$-c.c. and hence
      $((\delta^*)^+)^{\cR^*} = ((\delta^*)^+)^{\cR^*[x]}$. Let
      $\gamma = ((\delta^*)^+)^{\cR^*}$. Consider the countable set
      $\cR^*[x]|\gamma$ and the model $L[T, \cR^*[x]|\gamma]$. Since
      $x \in \cR^*[x]|\gamma$, we have that
      $L[T,x] \subseteq L[T, \cR^*[x]|\gamma]$. Now, Lemma
      \ref{lem:propL[T,a]} implies that every real $y$ in
      $L[T, \cR^*[x]|\gamma]$, belongs to an $\omega_1$-iterable
      $\cR^*[x]|\gamma$-premouse $N^y$. By taking an initial segment
      if necessary, we can assume that $N^y$ is sound and
      $\rho_\omega(N^y) \leq \gamma$. Since
      $L[T, x] \subseteq L[T, \cR^*[x]|\gamma]$, it suffices to show
      that every real $y$ in an $\omega_1$-iterable
      $\cR^*[x]|\gamma$-premouse $N^y$ belongs to $\cR^*[x]$.

      Let $\bar N^y = \cP^{N^y}(\cR^*|\gamma)$ be the
      $\cR^*|\gamma$-premouse obtained as the result of a
      $\cP$-construction above $\cR^*|\gamma$ inside $N^y$ in the
      sense of \cite{SchSt09} or \cite[Section 3]{St08b}. As the size
      of the extender algebra at $\delta^*$ is small, $\bar N^y$ is
      again a premouse and by definability of the forcing
      $\rho_\omega(\bar N^y) \leq \gamma$ (see for example
      \cite[Section 3]{St08b} for a similar argument). Therefore, the
      suitability of $\cR^*$ yields $\bar N^y \unlhd \cR^*$ and hence
      $N^y = \bar N^y[x] \unlhd \cR^*[x]$, where $\bar N^y[x]$ and
      $\cR^*[x]$ are construed as $\cR^*[x]|\gamma$-premice. Thus
      $y \in \cR^*[x]$, as desired.
    \end{proof}

    Since $\cR^*$ is countable in $L[T,x]$, in particular $\delta^*$,
    the Woodin cardinal in $\cR^*$, is countable in $L[T,x]$. Using
    the subclaim, this implies that $\delta^*$ is countable in
    $\cR^*[x]$. But the extender algebra at $\delta^*$ in $\cR^*$ has
    the $\delta^*$-c.c., so $\delta^*$ remains a cardinal in
    $\cR^*[x]$---a contradiction.
  \end{proof}

  Since all iteration trees are maximal, we have $\cR^- = \cM(\cT)$,
  where $\cT$ is an iteration tree of limit length on a suitable
  $a$-premouse coded by some real $z \leq_T x$. Then $\cR = Lp(\cR^-)$
  satisfies \eqref{eq:RNotProjAcross} and \eqref{eq:RdeltaWoodin}.

  For \eqref{eq:ODdeltainR}, note that
  $\mathcal{P}(\delta) \cap \od_{\cR^{-} \cup \{\cR^{-}\}}$ is the set
  of all subsets of $\delta$ which belong to an $\omega_1$-iterable
  $\cR^{-}$-premouse. Since $\cR = Lp(\cR^{-})$ these correspond to
  initial segments of $\cR$.

  This also implies that every subset of $a$ in $\od_{a \cup \{a\}}$
  belongs to $\cR$ since
  $\od_{a \cup \{a\}} \subseteq \od_{\cR^{-} \cup \{\cR^{-}\}}$. For
  the other inclusion in \eqref{eq:ODinR}, suppose that $b$ is a
  subset of $a$ and $b \in \cR$. As no new subsets of $a$ are added
  during the iteration,
  $b \in N$ for some short tree iterable, suitable $a$-premouse $N$
  coded by some real $z \leq_T x$. But then $N | \eta$ is
  $\omega_1$-iterable for every ordinal $\eta$ such that
  $\rho_\omega(N|\eta) = a$ as, in these cases, $\cQ$-structures
  exist. Hence, $b$ belongs to an $\omega_1$-iterable $a$-premouse and
  hence to $\od_{a \cup \{a\}}$.

  Finally, \eqref{eq:omega1delta} follows by the same argument as in
  the proof of the subclaim since the assumption that
  $\delta < \omega_1^{L[T,x]}$ together with \eqref{eq:RNotProjAcross}
  and \eqref{eq:RdeltaWoodin} suffices to derive the contradiction.
\end{proof}

The construction of $\cR_a^x$ depends only on the Turing degree of
$x$, i.e., $x \equiv_T y$ implies $\cR_a^x = \cR_a^y$. Therefore, we will
also write $\cR_a^d$ for $\cR_a^x$, if $d = [x]_T$.

\subsection{Prikry-like forcing a premouse}\label{subsec:Prikrypremouse}
We define a Prikry-like partial order $\bP$ to add a premouse with
infinitely many Woodin cardinals. Let $\mathcal{D}$ denote the set of
all Turing degrees and let $\mu$ denote the Martin measure on
$\mathcal{D}$. Further, let $\cD^m$ be the set of all increasing
sequences of Turing degrees of length $m$ and $\mu_m$ be the measure
on $\cD^m$ induced by the product of $\mu$. More precisely, let
$\mu_0 = \mu$ and assume inductively that $\mu_{k}$ is already defined
on $\cD^k$ for some $k < m$. Then we let for any $X \in \cD^{k+1}$,
$\mu_{k+1}(X) = 1$ if, and only if, for $\mu_0$-a.e. $d_0$ and
$\mu_k$-a.e.  $(d_1, \dots, d_k)$, $(d_0, \dots, d_k) \in X$.

To define the Prikry-like partial order $\bP$, we first define a
sequence of premice along an increasing sequence of Turing
degrees. For $\vec d = (d_0, \dots, d_m) \in \cD^{m+1}$, we let
\[ \cQ^{\vec d}_0(a) = \cR_a^{d_0} \] if $d_0$ is large enough
so that $\cR_a^{d_0}$ is defined, and recursively
\[ \cQ^{\vec d}_{i+1}(a) = \cR_{\cQ^{\vec d}_i(a)}^{d_{i+1}} \] for
$i < m$, if $d_{i+1}$ is large enough so that
$\cR_{\cQ^{\vec d}_i(a)}^{d_{i+1}}$ is defined.

Recall that $U$ is a $\Sigma^2_1$ set which is universal for
$\SIGMA^2_1$ in $V = M_1(A)$. Now, the conditions in $\bP$ are of the
form $(s, \vec X)$, where
\begin{enumerate}
\item $s = (\cS_0, \dots, \cS_n)$ is a sequence of premice such that
  for some $\vec d_s \in \cD^{n+1}$, $\cS_i= \cQ_i^{\vec d_s}(\emptyset)$
  for all $i \leq n$; and
\item $\vec X = (X_k \colon k < \omega) \in L(U,A)$\footnote{The reason for
    requiring $\vec X \in L(U,A)$ will become apparent in the proof of
    Lemma \ref{lem:agreement}.} is a sequence of sets such that for
  all $k < \omega$,
  \begin{enumerate}
  \item $X_k$ is a collection of $(k+1)$-sequences of premice, and
  \item
    $(\cQ^{\vec d}_{0}(\cS_n), \dots, \cQ^{\vec d}_{k}(\cS_n)) \in
    X_{k}$ for $\mu_{k+1}$-a.e. $\vec d \in \mathcal{D}^{k+1}$.
\end{enumerate}   
\end{enumerate}

We call $s$ the stem of the condition $(s, \vec X)$. For two
conditions $(s, \vec X)$ and $(r, \vec Y)$ in $\bP$ with
$s = (\cS_0, \dots, \cS_n)$ and $r = (\cR_0, \dots, \cR_m)$, we let
$(s, \vec X) \leq (r, \vec Y)$ if, and only if, one of the following holds:
\begin{enumerate}
\item $s = r$ and $X_i \subseteq Y_i$ for
all $i<\omega$; or 
\item for some $k < \omega$ and a sequence
  $(\cQ^{\vec d}_{0}(\cR_m), \dots, \cQ^{\vec d}_{k}(\cR_m)) \in Y_k$
  given by some $\vec d \in \cD^{k+1}$,
\begin{enumerate}
\item $s = r^{\frown} (\cQ^{\vec d}_{0}(\cR_m), \dots, \cQ^{\vec d}_{k}(\cR_m))$, and
\item for all $i < \omega$ and sequences $\vec e \in \cD^{i+1}$ such
  that
  $(\cQ^{\vec e}_{0}(\cS_n), \dots, \cQ^{\vec e}_{i}(\cS_n))$ is
  defined and belongs to $X_i$, we have
  \[ \big(\cQ^{\vec d^{\frown} \vec e}_{0}(\cR_m), \dots, \cQ^{\vec
      d^{\frown} \vec e}_{k}(\cR_m), \cQ^{\vec d^{\frown} \vec
      e}_{k+1}(\cR_m), \dots, \cQ^{\vec d^{\frown} \vec
      e}_{k+i+1}(\cR_m)\big) \in Y_{k+i+1}. \]
\end{enumerate}
\end{enumerate}

\comm{ We can actually restrict to the following dense suborder of
  $\bP$. We call a condition $(s, \vec X) \in \bP$ a \emph{measure-one
    tree} if, and only if, $\vec X = (X_k \colon k < \omega)$, where each
  $X_k$ is a set of stems each of which is compatible with
  $s = (\cS_0, \dots, \cS_n)$ and the following hold:
\begin{enumerate}
\item $\vec X$ is closed under initial segments in the following sense:
whenever $(x_0, \dots, x_k) \in X_k$ for some $k < \omega$,
  then $(x_0, \dots, x_i) \in X_i$ for all $i < k$; and
\item for all $k<\omega$ and all $x \in X_k$, 
  \[ \Big\{d\in\mathcal{D}:x^{\frown}(\cR_{\cS_n}^{d}) \in X_{k+1}\Big\} \in \mu.\]
\end{enumerate}}

The next lemma shows that $\bP$ has the Prikry property. As the proof
is analogous to e.g., the proof of Corollary 6.39 in \cite{StW16}, we
omit it.

\begin{lemma}\label{lem:PrikryProperty}
  Let $(s, \vec X) \in \bP$ be a condition and $\Lambda$ a countable
  set of sentences in the forcing language. Then there is some
  $(s, \vec Y) \leq (s, \vec X)$ such that $(s, \vec Y)$ decides
  $\phi$, for all $\phi \in \Lambda$. 
\end{lemma}

Now fix a $G$ which is $\bP$-generic over $M_1(A)$ and let
$\vec \cQ = (\cQ_n \colon n < \omega)$ be the union of the stems of
conditions in $G$. Write $\delta_n$ for the largest cardinal in
$\cQ_n$. By definition, all $\cQ_n$ are such that
$\cQ_n = Lp(\cQ_n|\delta_n)$ and $Lp^{\omega}(\cQ_n|\delta_n)$ is a
suitable premouse, so $\delta_n$ is a Woodin cardinal in $\cQ_n$. Let
\[ \cQ_\infty = \bigcup_{n<\omega} \cQ_n. \]

\begin{lemma}\label{lem:propQn}
  The following hold:
  \begin{enumerate} 
  \item for all $n<\omega$, $\Pot(\delta_n) \cap L[\vec \cQ] \subseteq
    \cQ_n$, \label{eq:propQn1}
  \item for all $n < \omega$, $\delta_n$ is a Woodin cardinal in
    $L[\vec \cQ]$, \label{eq:propQn2}
  \item $\cQ_n = \cQ_\infty | (\delta_n^+)^{\cQ_\infty}$; hence,
    $L[\cQ_\infty] = L[\vec \cQ]$. \label{eq:propQn3}
  \end{enumerate}
\end{lemma}

\begin{proof}
  We show \eqref{eq:propQn1}, from which \eqref{eq:propQn2} and
  \eqref{eq:propQn3} follow.  Let us first show that
  $\Pot(\delta_n) \cap \cQ_{n+1} \subseteq \cQ_n$; a consequence of
  this is that $\Pot(\delta_n) \cap \cQ_m \subseteq \cQ_n$ whenever
  $n<m$.  To see this, suppose there is a subset $a$ of $\delta_n$
  which is in $\cQ_{n+1}$. By Lemma \ref{lem:propRax}\eqref{eq:ODinR},
  $\Pot(\delta_n) \cap \cQ_{n+1} = \Pot(\delta_n) \cap \od_{\cQ_n \cup
    \{\cQ_n\}}$.
  Lemma \ref{lem:propRax}\eqref{eq:ODdeltainR} implies that
  $\Pot(\delta_n) \cap \cQ_{n} = \Pot(\delta_n) \cap \od_{\cQ_n^- \cup
    \{\cQ_n^-\}}$, but this is equal to
  $\Pot(\delta_n) \cap \od_{\cQ_n \cup \{\cQ_n\}} $ by definability of
  $\cQ_n = Lp(\cQ_n^-)$. Hence, we have $a \in \cQ_n$, as desired.


  To prove \eqref{eq:propQn1}, let
  $a \in \Pot(\delta_n) \cap L[\vec \cQ]$. Let $\dot a$ be a term
  defining $a$ from $\vec \cQ$ and an ordinal parameter in
  $M_1(A)[G]$. The Prikry property (Lemma \ref{lem:PrikryProperty})
  yields a $k<\omega$ and a condition $(s, \vec X)$ with $s$ of the
  form $(s_0, \dots, s_k)$, with $n<k$, which decides all statements
  of the form ``$\xi \in \dot a$''. By genericity we can choose
  $(s, \vec X) \in G$, so that, in particular, $s_i = \cQ_i$ for all
  $i \leq k$.
  \begin{claim} We have $\xi \in a$ if, and only if,
    \[\exists t \exists \vec Y (t \text{ is of
      the form } (t_0, \dots, t_k) \, \wedge \, t_i = \cQ_i \text{ for all }
    i \leq k \, \wedge \, (t, \vec
    Y) \, \Vdash \, \xi \in \dot a). \] 
  \end{claim}
  \begin{proof}
    If $\xi \in a$, then the condition $(s, \vec X)$ is a witness for
    the displayed equation. Conversely, suppose there is some
    condition $(t, \vec Y)$ as in the displayed equation, but
    $\xi \notin a$. Then we must have
    $(s, \vec X) \, \Vdash \, \xi \notin \dot a$. Note that $t =
    s$. Define $\vec Z$ by $Z_i = X_i \cap Y_i$ for all $i <
    \omega$. Then $(s, \vec Z) \leq (s, \vec X)$ and
    $(s, \vec Z) = (t, \vec Z) \leq (t, \vec Y)$. Now, let $H$ be
    $\bP$-generic over $M_1(A)$ such that $H$ contains $(s, \vec
    Z)$. Then, in $M_1(A)[H]$, both $\xi \in a$ and $\xi \notin a$
    hold, a contradiction.
  \end{proof}
  The claim yields that $a \in \od_{\cQ_k \cup \{\cQ_k\}}$ and thus
  $a \in \od_{\cQ_k^- \cup \{\cQ_k^-\}} \cap \Pot(\delta_k) = \cQ_k
  \cap \Pot(\delta_k)$, by Lemma
  \ref{lem:propRax}\eqref{eq:ODdeltainR}. So in particular
  $a \in \cQ_k \cap \Pot(\delta_n)$. By the remark at the beginning of
  the proof, this implies $a \in \cQ_n$, as desired.
\end{proof}

Write $\lambda = \sup_{n<\omega} \delta_n$ and fix any premouse
$\cP\models\zfc$ extending $\cQ_\infty$ in which $\lambda$ remains a
cardinal. We form a derived model of $\cP$. More precisely, working in
$M_1(A)[G]$, let $\bS$ be the partial order consisting of sequences
$(h_0, \dots, h_k)$ such that for all $0 \leq n \leq k$,
$h_n \in M_1(A)$ is $\col(\omega, \delta_n)$-generic over
$\cQ_n$. The order on $\bS$ is sequence extension. Fix $\hat h$
$\bS$-generic over $M_1(A)[G]$, let $(h_n \colon n < \omega)$ be the
induced sequence, and let $h$ be given by $h(n,m) = h_n(m)$. We may
abuse notation and identify $\hat h$ with $h$ and with the
corresponding $\col(\omega, {<}\lambda)$-generic filter over
$M_1(A)[G]$, since $(\delta_n \colon n < \omega)$ is definable from
$\cQ_\infty$ by the previous lemma.

Using this $h$, we can build the derived model of $\cP$: Write
\[\bR^*_h = \bigcup_{n\in\omega}
  \bR \cap \cP[h\upharpoonright\delta_n],\] and
\begin{align*}
  Hom^*_h = \{ p[S] &\cap \bR^*_h | \exists n < \omega (\cP[h
  \upharpoonright \delta_n] \vDash S \text{ is a } \\ & {<}\lambda
  \text{-absolutely complemented tree})\}. 
\end{align*}

Note that $\bR^*_h$, and $\Hom^*_h$ only depend on $\cQ_\infty$ and
$h$, not on the full premouse $\cP$, so this notation makes sense. For
this reason, we sometimes do not distinguish between $\cP$ and
$\cQ_\infty$ in what follows.

\subsection{Preparation for adding extenders on
  top}\label{subsec:prepaddext}

Our next goal is to \emph{add the extenders of $M_1(A)$} on top of
$\cQ_\infty$ while preserving the Woodin cardinals of $\cQ_\infty$ to
obtain a model with $\omega +1$ Woodin cardinals.  This will be done
via a $\cP$-construction. For this, some preparation is needed--- we
need to show e.g.,
\[L_\xi[\cQ_\infty][h] = (M_1(A)|\xi)[G][h]\] for some ordinal $\xi$
below the Woodin cardinal of $M_1(A)$.

We first need the following lemmas which again essentially can be
found also in \cite{St}. Recall that we have $V = M_1(A)$ and write
$\bR^V = M_1(A) \cap \bR = A$.

\begin{lemma}\label{LemmaRstarisRV}
$\mathbb{R}^*_h = \mathbb{R}^V$.
\end{lemma}
\begin{proof}
  $\mathbb{R}^*_h \subseteq \mathbb{R}^V$ is easy to see as each pair
  $(\cQ_n, h_n)$ is in $M_1(A)$, so let $x \in \bR^V$ for the other
  inclusion. Let $(s, \vec X) \in \bP$ be an arbitrary condition in
  the Prikry-like forcing defined above, say
  $s = (\cS_0, \dots, \cS_k)$. For each $n<\omega$ consider
  \[ Y_n = \{ t \in X_n \mid \exists \vec d \in \cD^{n+1} (x \leq_T
    d_0 \wedge t = (\cQ_0^{\vec d}(\cS_k), \dots, \cQ_n^{\vec
      d}(\cS_k)))\} \] and note that $(s, \vec Y) \leq (s, \vec
  X)$. Moreover, since pseudo-genericity iterations are included in
  the construction of the premice $\cR_a^d$, it follows by density
  that for all $i >k$, $x$ is generic over $\cQ_i$ for Woodin's
  extender algebra at $\delta_i$. Hence $x \in \bR^*_h$.
\end{proof}

\begin{lemma}\label{LemmaHomstarIsDelta21}
  $Hom^*_h = \DELTA^2_1$. 
\end{lemma}
\begin{proof}
  We first show $\DELTA^2_1\subseteq Hom^*_h$. Let $B \in \DELTA^2_1$,
  so that $B$ is, since $\SIGMA^2_1$ has the scale property, Suslin
  and co-Suslin. More precisely, choose $z\in\mathbb{R}$ such that
  $B \in \Delta^2_1(z)$ and choose $i \in\omega$ large enough so that
  $z \in \cQ_i[h_i]$ (this exists by the previous lemma). From $T$ and
  $z$, one can construct trees $S_0$ and $S_1$ on
  $\omega\times\kappa$, for some ordinal
  $\kappa<\boldsymbol\delta^2_1$, such that $S_0$ and $S_1$ project to
  $B$ and $\mathbb{R}\setminus B$, respectively.  We have
  \[S_0, S_1 \in L[T, \cQ_j, h_i]\] for every $j\geq i$. Moreover, by
  the absoluteness of well-foundedness, if
  $x \in L[T, \cQ_j, h_i]\cap\mathbb{R}$, then
  \[L[T, \cQ_j, h_i]\models x \in p[S_0]\cup p[S_1].\]

  Moreover, if $\mathbb{Q}$ is a small partial order in
  $L[T, \cQ_j, h_i]$ (say, of size strictly less than $\delta_j$) and
  $g$ is $\mathbb{Q}$-generic over $L[T, \cQ_j, h_i]$, then
  \[L[T, \cQ_j, h_i][g]\models \mathbb{R} = p[S_0]\cup p[S_1],\] for
  (since $\delta_j$ is countable in $V$) otherwise there is a
  $\mathbb{Q}$-generic $g$ over $L[T, \cQ_j, h_i]$ with $g\in V$ such
  that
  \[L[T, \cQ_j, h_i][g]\models \exists x\in\mathbb{R}\, \big(x \not\in
    p[S_0]\cup p[S_1]\big).\] However, such an $x$ does belong to one
  of $p[S_0]$ or $p[S_1]$ in $V$ and thus it must do too in
  $L[T, \cQ_j, h_i][g]$ (by the absoluteness of well-foundedness).

  We have shown that in $L[T, \cQ_j, h_i]$ there are
  ${<}\delta_j$-absolutely complementing trees $S_0$ and $S_1$ that
  project to $B \cap L[T, \cQ_j, h_i]$ and
  $(\mathbb{R}\setminus B) \cap L[T, \cQ_j, h_i]$, respectively.

  The trees $S_0$ and $S_1$ might be big (in principle $\kappa$ might
  be of arbitrarily large size below $\boldsymbol\delta^2_1$). Let $M$
  be an elementary substructure of some large $L_\alpha[T,\cQ_j,h_i]$
  such that
  \begin{enumerate}
  \item $M \in L[T,\cQ_j,h_i]$,
  \item $L[T,\cQ_j,h_i] \models |M| = \delta_j$,
  \item $L_{\delta_j}[T,\cQ_j,h_i] \subseteq M$, and
  \item $S_0, S_1 \in M$.
  \end{enumerate}
  Let $S_0^j$ and $S_1^j$ be the images of $S_0$ and $S_1$ under the
  collapse embedding for $M$. Then, $S_0^j$ and $S_1^j$ are also
  ${<}\delta_j$-absolutely complementing and $S_0^j$ projects to
  $B \cap L[T, \cQ_j, h_i][g]$ whenever $g$ is $\mathbb{Q}$-generic
  over $L[T, \cQ_j, h_i]$ for a partial order $\mathbb{Q}$ in
  $L[T,\cQ_j,h_i]$ of size strictly less than $\delta_j$.

  \begin{claim}
    For all $j >i$, $S_0^j, S_1^j \in \cQ_\infty[h_i]$.
  \end{claim}
  \begin{proof}
    Note that $h_i$ is $\col(\omega, \delta_i)$-generic over
    $\cQ_{j+1}$.
    Recall that by definition, $\cQ_{j+1} = \cR^{d}_{\cQ_j}$ for some
    Turing degree $d$. Therefore we have, for such a $d$, by an
    argument similar to that in the proof of Subclaim
    \ref{SubclaimR*[x]} that
    \[ \cQ_{j+1}[h_i] = \cR^{d}_{\cQ_j[h_i]}. \]

    Now, by Lemma \ref{lem:propL[T,a]}, every subset of $\cQ_j$ in
    $L[T,\cQ_j,h_i]$ is definable from $\cQ_j$, $h_i$, and ordinal
    parameters in $V$. $S_0^j$ and $S_1^j$ are essentially subsets of
    $\delta_j$ in $L[T,\cQ_j,h_i]$, so they are definable in $V$ from
    $\cQ_j$, $h_i$, and ordinal parameters. By Lemma
    \ref{lem:propRax}(\ref{eq:ODinR}),
    \[ \mathcal{P}(\cQ_j[h_i])\cap \cR^d_{\cQ_j[h_i]} =
      \mathcal{P}(\cQ_j[h_i])\cap \od_{\cQ_j[h_i] \cup
        \{\cQ_j[h_i]\}}. \] This implies that $S_0^j$ and $S_1^j$
    belong to $Q_{j+1}[h_i]$, which proves the claim.
  \end{proof}

  It follows that $S_0^j$ and $S_1^j$ are ${<}\delta_j$-absolutely
  complementing trees in $\cQ_\infty[h_i]$ since by an argument as in
  the proof of Lemma \ref{lem:propQn}, for every $g$ which is generic
  for a partial order of size strictly less than $\delta_j$,
  $\cQ_\infty[g] \cap \Pot(\delta_j) \subseteq \cQ_j[g]$. In
  particular, $B \cap \cQ_\infty[h_i]$ is ${<}\lambda$-universally
  Baire in $\cQ_\infty[h_i]$. The canonical extension $B^*$ of
  $B \cap \cQ_\infty[h_i]$ to a set in $Hom^*_h$ is in fact unique and
  consistent with the trees $(S_0^j,S_1^j)$ for all
  $i<j<\omega$. Therefore $B = B^*$. This proves
  $\DELTA^2_1\subseteq Hom^*_h$.

  We now assume towards a contradiction that $\DELTA^2_1\neq Hom^*_h$,
  i.e., that the $\DELTA^2_1$ sets form a proper Wadge initial segment
  of $Hom^*_h$. Since $Hom^*_h$ is closed under continuous
  reducibility and the $\Sigma^2_1$ set $U$ which is universal for
  $\SIGMA^2_1$ is minimal in the Wadge hierarchy above the pointclass
  $\DELTA^2_1$, it follows that $U \in Hom^*_h$.

  By definition of $Hom^*_h$, there is some $i\in\omega$ and a
  ${<}\lambda$-absolutely complemented tree $S$ in $\cQ_\infty[h_i]$
  such that, since $\bR^V = \bR^*_h$ by Lemma \ref{LemmaRstarisRV},
  $p[S] \cap \bR^V = \bR^V \setminus U$.

  By Lemma \ref{lem:propL[T,a]}, the relation
  \[x \not\in\od_{\{y\}}\] is $\Pi^2_1$, so it appears in a section of
  the complement of $U$. Using $S$, we can get a real $x$ such that
  $x \notin \od_{\{(\cQ_i|\delta_i, h_i)\}}$ but
  $x \in \cQ_\infty[h_i]$. Then in fact $x \in \cQ_i[h_i]$ by the
  argument in the proof of Lemma \ref{lem:propQn}. However,
  $\cQ_i = Lp(\cQ_i|\delta_i),$ so every real in $\cQ_i[h_i]$ is
  definable from $\cQ_i|\delta_i$, $h_i$, and ordinal parameters,
  which is a contradiction. This completes the proof of the lemma.
\end{proof}

Recall that $\bR^V = A$. Define $\xi_0$ to be the least
$\xi > \Theta^{L(U,A)}$ such that $L_{\xi}(U, A) \vDash \zf$. Since
$V | \xi_0$ is a countably iterable $A$-premouse, it is easy to see
that
\[ V|\xi_0 = L_{\xi_0}(U, A). \]

Finally, we have the following agreement between $L[\cQ_\infty][h]$
and $L(U, A)[G][h]$ (as classes). Here we will use the extra condition
``$\vec X \in L(U,A)$'' in the definition of the Prikry-like forcing
$\bP$ defined in Section \ref{subsec:Prikrypremouse} to ensure that
$\bP \in L(U,A)$.

\begin{lemma}\label{lem:agreement}
  $L[\cQ_\infty][h] = L(U, A)[G][h]$.
\end{lemma}
\begin{proof}
  In $L[\cQ_\infty][h]$, one can easily compute the derived model of
  $L[\cQ_\infty]$ associated to $h$. Thus, using Lemma
  \ref{LemmaHomstarIsDelta21}, it follows that
  $U \in L[\cQ_\infty][h]$. By Lemma \ref{LemmaRstarisRV},
  $A = \mathbb{R}^V \in L[\cQ_\infty][h]$ and, using $L(U,A)$ and
  $\cQ_\infty$, one can easily define $\mathbb{P}$ and
  $G$. Conversely, from $G$ one can easily recover $\cQ_\infty$.
\end{proof}

This, together with the observation above, has the following
corollary:

\begin{corollary}\label{cor:prepPconstr}
  $L_{\xi_0}[\cQ_\infty][h] = V|\xi_0[G][h]$.
\end{corollary}

\subsection{Adding extenders on top}\label{SubsectionPConst}

Finally, we use a $\cP$-construction (see for example \cite[Section
3]{St08b} or \cite{SchSt09}) to add extenders witnessing another
Woodin cardinal on top of $\cQ_\infty$. In Lemma
\ref{lem:prepPconstr}, we first extend Corollary \ref{cor:prepPconstr}
to ordinals $\xi > \xi_0$ in order to obtain an appropriate background
universe $W$ for the $\cP$-construction. Using the fact that $V[G][h]$
is a forcing extension of the $A$-premouse $V = M_1(A)$ by a small
forcing, the proof of this lemma is straightforward and similar to the
argument in \cite[Section 3]{St08b}, so we omit it.

\begin{lemma}\label{lem:prepPconstr}
  There is a proper class $(\cQ_\infty, h)$-premouse $W$ such that for
  any $\xi \geq \xi_0$, 
  \begin{enumerate}
  \item $W|\xi$ has the same universe as $V|\xi[G][h]$,
  \item for any $k < \omega$, $\rho_k(W|\xi) = \omega$ if, and only if,
    $\rho_k(V|\xi) = A$, and
  \item for any $k < \omega$, if $\rho_k(W|\xi) > \omega$, then
    $\rho_k(W|\xi) = \rho_k(V|\xi)$ and
    $p_{k+1}(W|\xi) = p_{k+1}(V|\xi)$.
  \end{enumerate}
\end{lemma}

Let $\cP = \cP^{W}(L_{\xi_0}[\cQ_\infty])$ be the result of a
$\cP$-construction above $L_{\xi_0}[\cQ_\infty]$ performed inside $W$
and let $\cP_\xi$ for $\xi \geq \xi_0$ denote the levels of the
$\cP$-construction. The following lemma shows that $\cP$ is as
desired.

\begin{lemma}\label{lem:PconstrSuccessful}
The following hold:
  \begin{enumerate}
  \item For $\xi \geq \xi_0$, if $X \subseteq \cP_\xi$ is definable
    over $\cP_\xi$ with parameters from $\cP_\xi$, then
    $X \cap \cQ_\infty | \delta_n \in \cQ_\infty$ for all
    $n < \omega$. \label{eq:PconstrSuccessful1}
  \item For all $\xi \geq \xi_0$, $\rho_\omega(\cP_\xi) \geq
    \lambda$. \label{eq:PconstrSuccessful2}
  \item $\cP$ is a premouse and $\cP[h] =
    W$. \label{eq:PconstrSuccessful4}
  \item $\cP$ has $\omega+1$ Woodin
    cardinals. \label{eq:PconstrSuccessful5}
  \end{enumerate}
\end{lemma}

\begin{proof}
  For the proof of \eqref{eq:PconstrSuccessful1} note that $\cP_\xi$
  is 
  definable over $V|\xi[G]$ from $\cQ_\infty$ since the translation
  between $V|\xi[G][h]$ and $W|\xi$ is definable and $h$ is generic
  over $V|\xi[G]$ for a homogeneous forcing. Hence, if
  $X \subseteq \cP_\xi$ is definable over $\cP_\xi$ with parameters
  from $\cP_\xi$, then $X \in V[G]$. We can assume without loss of
  generality that $X$ is a set of ordinals and let $\dot X_n$ for
  every $n < \omega$ be a term defining $X \cap \delta_n$ from
  $\cQ_\infty$ and the ordinal parameter $\xi$. Now we can argue as in
  the proof of Lemma \ref{lem:propQn} to obtain
  $X \cap \delta_n \in \cQ_\infty$, as desired.

  Suppose \eqref{eq:PconstrSuccessful2} fails and let $\xi \geq \xi_0$
  be least such that $\rho_\omega(\cP_\xi) = \rho < \lambda$. Say this
  is witnessed by some set of ordinals $X \subseteq \rho$, which is
  definable over $\cP_\xi$ with parameters from $\cP_\xi$, but
  $X \notin \cP_\xi$. Since $\lambda = \bigcup_{n < \omega} \delta_n$,
  there is some $n < \omega$ such that $\rho < \delta_n$. This means
  that $X = X \cap \delta_n \in \cQ_\infty$ by
  \eqref{eq:PconstrSuccessful1}. But $\cQ_\infty \subseteq \cP_\xi$, a
  contradiction.

  Finally, \eqref{eq:PconstrSuccessful4} and
  \eqref{eq:PconstrSuccessful5} now follow from this and standard
  properties of the $\cP$-construction, see for example
  \cite{SchSt09}.
\end{proof}

This finishes the proof of Theorem \ref{TheoremLongIntro} and, by
putting the active extender of $M_1^\sharp(A)$ on an initial of $\cP$,
e.g., as in \cite[Section 2]{FNS10}, also the proof of Theorem
\ref{thm:omega+1Wdns}.

\section{Applications}\label{SectionRemaining}

We sketch a proof of the equiconsistencies of the schemata stated in
Corollary \ref{TheoremLongPD}. 

\begin{proof}[Proof of Corollary \ref{TheoremLongPD}]
  We start with the equiconsistency of \eqref{eq:2}, \eqref{eq:1.5},
  and \eqref{eq:1}. Theorem \ref{TheoremADM1} together with Theorem
  \ref{thm:M1DC} immediately gives that the consistency of
  \eqref{eq:2} implies the consistency of \eqref{eq:1.5}. Now suppose
  that \eqref{eq:1.5} is consistent, say there is a model $M$ of
  $\zf + \dc + \ad$ with $n+5$ Woodin cardinals. Using a fully
  backgrounded extender construction as in \cite{MS94}, it is easy to
  see that $M_n^\sharp(\bR^M)$ exists in $M$ and
  $M_n(\bR^M) \cap \bR = \bR^M$. Moreover, as winning strategies for
  games of length $\omega$ on $\omega$ can be coded by reals,
  $M_n(\bR^M) \models \ad$. Therefore the argument in Section
  \ref{SectOmega+1WdnsFromAD} shows that there is a model of $\zfc$
  with $\omega + n$ Woodin cardinals. The direction from \eqref{eq:1}
  to \eqref{eq:2} follows from \cite{Ne04}, where Neeman argues in
  Appendix A, using results of \cite{MaSt94}, that the existence of a
  model with $\omega+n$ Woodin cardinals and a measurable cardinal
  above them all implies the existence of a sufficiently iterable
  active premouse with $\omega+n$ Woodin cardinals to which Theorem
  2A.3 in \cite{Ne04} applies.

  Next, we argue that \eqref{eq:3} is equiconsistent with \eqref{eq:2}
  and \eqref{eq:1.5}. First, suppose that there is a model $V$ of
  $\zf + \dc + \ad$ and, say, $\PI^1_{n+5}$-determinacy for games of
  length $\omega$ on $\bR$ and work in this model. The methods of
  Section \ref{SectADinM1} can be used to conclude from
  $\PI^1_{n+5}$-determinacy for games of length $\omega$ on reals that
  $M_n^\sharp(\bR)$ exists and is countably iterable (see \cite{AgMu}
  for details). Since $M_n^\sharp(\bR)$ contains all the reals, it in
  particular contains codes for all winning strategies for games on
  natural numbers of length $\omega$ which exist in $V$. Therefore
  $M_n^\sharp(\bR) \vDash \ad$ witnesses \eqref{eq:1.5}. For the other
  direction, suppose there is model of $\zfc$ in which, say,
  $\PI^1_{n+5}$-determinacy holds for games of length $\omega^2$ on
  $\omega$. By the results in Section \ref{SectADinM1}, there a
  countable set of reals $A$ such that
  $M_{n+4}(A) \vDash \text{``}\bR = A\text{''} + \ad$. Now work inside
  $V = M_{n+4}(A)$ and write $\bR = A$. Note that $V$ is a model of
  $\zf+\dc$ (see e.g., \cite{Mu}). Since $M_{n+3}^\sharp(\bR)$ exists
  in $V$, it follows from the proof of projective determinacy in
  \cite{Ne10} that $\PI^1_n$-determinacy holds for games of length
  $\omega$ on $\bR$ (see \cite{AgMu} for details).

  Finally, we argue that \eqref{eq:4} is equiconsistent with the other
  schemata. Suppose there is model of $\zfc$ in which, say,
  $\PI^1_{n+5}$-determinacy holds for games of length $\omega^2$ on
  $\omega$. Then again by the results in Section \ref{SectADinM1}
  there a countable set of reals $A$ such that
  $M_{n+4}(A) \vDash \text{``}\bR = A\text{''} + \ad$ and we can work
  in $V = M_{n+4}(A)$. As $V$ has sufficiently many Woodin cardinals,
  every $\col(\omega, \bR)$-generic extension $V[g]$ is closed under
  $x \mapsto M_{n}^\sharp(x)$ for all reals $x$ and hence satisfies
  $\PI^1_{n+1}$-determinacy for games of length $\omega$ on
  $\omega$. For the other direction we show the consistency of
  \eqref{eq:3}. Suppose that there is a model $V$ of $\zf+\dc+\ad$
  such that, say, $\PI^1_{n+5}$-determinacy for games of length
  $\omega$ on $\omega$ holds in every $\col(\omega,\bR)$-generic
  extension of $V[g]$ of $V$. As $\bR^V$ is a countable set of reals
  in $V[g]$, $M_{n+1}^\sharp(\bR^V)$ exists in $V[g]$. By homogeneity
  of the forcing, this implies that in $V$ there is a set of reals
  coding a model with $n+1$ Woodin cardinals which contains
  $\bR^V$. This suffices to run the argument in \cite{Ne10} mentioned
  above which yields $\PI^1_n$-determinacy for games of length
  $\omega$ on $\bR$.
\end{proof}

We finish this article by sketching the modifications which are needed
to prove Theorems \ref{TheoremLongomega+omegaIntro} and
\ref{TheoremLongSalphaIntro}.

\begin{proof}[Proof of Theorem \ref{TheoremLongomega+omegaIntro}]
  Instead of Theorem \ref{TheoremADM1}, we now aim to obtain a
  countable set of reals $A$ such that $M_\omega(A) \cap \bR = A$ and
  $M_\omega(A) \vDash \zf + \ad$. For this purpose, we replace
  1-smallness in the definition of $\varphi$-witness (see Definition
  \ref{def:phiwitness}) by $\omega$-smallness. Moreover, we replace
  $\Pi^1_2$-iterability in the definition of the game
  $\cG_{\varphi,\psi}$ (see Definition \ref{def:Gvarphipsi}) by the
  notion of \emph{weak iterability} (also called
  $\Game^\bR\Pi^1_1$-iterability) as in \cite[Definition
  7.7]{St10}. Note that determinacy of games of length $\omega^2$ with
  $\Game^\bR\Pi^1_1$ payoff suffices to show that this game is
  determined. Furthermore, Theorem 7.10 in \cite{St10} suffices to
  carry out the comparison arguments used to prove Theorem
  \ref{TheoremADM1}. The proofs in Sections \ref{SectionDC} and
  \ref{SectOmega+1WdnsFromAD} straightforwardly generalize to this
  context.
\end{proof}

\begin{proof}[Proof of Theorem \ref{TheoremLongSalphaIntro}]
  Instead of Theorem \ref{TheoremADM1}, we now aim to obtain a
  countable set of reals $A$ such that there is an $A$-premouse $M$ of
  class $S_\alpha$ such that $M \cap \bR = A$ and
  $M \vDash \zf + \ad$. For this purpose, we replace 1-smallness in
  the definition of $\varphi$-witness (see Definition
  \ref{def:phiwitness}) by not being of class $S_\alpha$. Moreover, we
  replace $\Pi^1_2$-iterability in the definition of the game
  $\cG_{\varphi,\psi}$ (see Definition \ref{def:Gvarphipsi}) by the
  notion of $\Pi^1_\alpha$-iterability as in \cite[Definition
  4.1]{Ag}. Note that determinacy of games of length $\omega^2$ with
  $\sigma$-projective payoff suffices to show that this game is
  determined. Furthermore, Lemmas 2.19 and 4.3 in \cite{Ag} suffice to
  carry out the comparison arguments used to prove Theorem
  \ref{TheoremADM1}. The proofs in Sections \ref{SectionDC} and
  \ref{SectOmega+1WdnsFromAD} straightforwardly generalize to this
  context.
\end{proof}


\bibliographystyle{abstract}
\bibliography{References}

\end{document}